\documentclass{amsart}

\usepackage{amssymb}
\usepackage{amsmath}
\usepackage{amsthm}
\usepackage{amsfonts}

\usepackage{a4wide}
\usepackage{longtable}

\newtheorem{theorem}{Theorem}[section]
\newtheorem{corollary}{Corollary}[theorem]

\theoremstyle{definition}
\newtheorem{definition}{Definition}[section]
\newtheorem{remark}{Remark}[section]
\newtheorem{example}{Example}[section]

\begin{document}

\title{Transfers of metabelian \(p\)-groups}

\author{Daniel C. Mayer}
\address{Naglergasse 53\\8010 Graz\\Austria}
\email{algebraic.number.theory@algebra.at}
\urladdr{http://www.algebra.at}
\thanks{Research supported by the
Austrian Science Fund,
Grant Nr. J0497-PHY}

\subjclass[2000]{Primary 20D15, 20F12, 20F14; Secondary 11R29, 11R11}
\keywords{Metabelian \(p\)-groups of maximal class, transfers of \(2\)-groups,
tree of metabelian \(3\)-groups of non-maximal class,
principalisation of \(p\)-class groups, quadratic base fields}

\date{November 15, 2010}

\begin{abstract}
Explicit expressions for the transfers \(\mathrm{V}_i\)
from a metabelian \(p\)-group \(G\) of coclass \(\mathrm{cc}(G)=1\)
to its maximal normal subgroups \(M_i\) \((1\le i\le p+1)\)
are derived by means of relations for generators.
The expressions for the exceptional case \(p=2\) differ significantly
from the standard case of odd primes \(p\ge 3\).
In both cases the transfer kernels \(\mathrm{Ker}(\mathrm{V}_i)\) are calculated
and the principalisation type of the metabelian \(p\)-group is determined,
if \(G\) is realised as the Galois group \(\mathrm{Gal}(\mathrm{F}_p^2(K)\vert K)\)
of the second Hilbert \(p\)-class field \(\mathrm{F}_p^2(K)\)
of an algebraic number field \(K\).
For certain metabelian \(3\)-groups \(G\)
with abelianisation \(G/G^{\prime}\) of type \((3,3)\)
and of coclass \(\mathrm{cc}(G)=r\ge 3\),
it is shown that the principalisation type determines
the position of \(G\) on the coclass graph \(\mathcal{G}(3,r)\)
in the sense of Eick and Leedham-Green.
\end{abstract}

\maketitle


\section{Introduction}
\label{s:Intro}

In \cite{Ma}
we have used the theory of dihedral fields of degree \(2p\) with a prime \(p\ge 3\)
to show for a quadratic base field \(K=\mathbb{Q}(\sqrt{D})\)
with \(p\)-class group \(\mathrm{Cl}_p(K)=\mathrm{Syl}_p(\mathrm{Cl}(K))\) of type \((p,p)\)
and second \(p\)-class group \(G=\mathrm{Gal}(\mathrm{F}_p^2(K)\vert K)\) of maximal class
that the entire \(p\)-class group of \(K\) becomes principal
in at least \(p\) unramified cyclic extension fields
\(N_i\vert K\) of relative degree \(p\).

In the present paper we generalise this result
for an arbitrary base field \(K\)
with \(p\)-class group \(\mathrm{Cl}_p(K)\) of type \((p,p)\)
and second \(p\)-class group \(G\) of maximal class.
The proof consists of
computing the images and kernels of the transfers \(\mathrm{V}_i\)
from \(G\) to its maximal normal subgroups \(M_i\) with \(1\le i\le p+1\)
and thus determining the principalisation type of \(K\).
Our statements are mainly expressed in group theoretical form,
using transfer types instead of principalisation types,
since the realisation of the \(p\)-group \(G\)
as Galois group \(\mathrm{Gal}(\mathrm{F}_p^2(K)\vert K)\)
of the second Hilbert \(p\)-class field of an algebraic number field \(K\),
though being fundamental for the number theoretical applications,
is inessential for proving our results.

Section \ref{s:MaxTrf} is devoted to
the metabelian \(p\)-groups \(G\) of maximal class
for an arbitrary prime \(p\ge 2\).
A \(p\)-group \(G\) 
of order \(\lvert G\rvert=p^m\) and nilpotency class \(\mathrm{cl}(G)=m-1\)
is called of maximal class or of coclass \(\mathrm{cc}(G)=m-\mathrm{cl}(G)=1\).
That is a CF-group
with first factor \(G/\gamma_2(G)\) of type \((p,p)\)
and cyclic further factors
\(\gamma_j(G)/\gamma_{j+1}(G)\) of order \(p\) for \(2\le j\le m-1\),
where we denote by \(\gamma_j(G)\) the members of the lower central series of \(G\).
Relations for the generators \(x,y\)
of a metabelian \(p\)-group \(G=\langle x,y\rangle\) of maximal class
have been given by N. Blackburn \cite{Bl} and
more generally by R. J. Miech \cite{Mi}
and are recalled at the beginning of section \ref{s:MaxTrf}.

In subsection \ref{ss:MaxTrmTrf}
we derive explicit expressions for the transfers
\(\mathrm{V}_i:G/\gamma_2(G)\to M_i/\gamma_2(M_i)\)
from \(G\) to the maximal self-conjugate subgroups \(M_1,\ldots,M_{p+1}\)
with the aid of commutator calculus and the presentations given by Miech.
The expressions obtained for the exceptional case \(p=2\) differ significantly
from the uniform standard case of odd primes \(p\ge 3\).

Subsection \ref{ss:TypTrf}
introduces the group theoretical concepts of
singulets and multiplets of transfer types,
which describe the transfer kernels.
In the intermediate subsection \ref{ss:ExtPrc}
we use the reciprocity law of E. Artin \cite{Ar1,Ar2}
to show how the group theoretical statements
can be applied to algebraic number fields \(K\)
with \(p\)-class group \(\mathrm{Cl}_p(K)\) of type \((p,p)\).
The \textit{second \(p\)-class group} of such a field \(K\),
that is the Galois group \(G=\mathrm{Gal}(\mathrm{F}_p^2(K)\vert K)\)
of the second Hilbert \(p\)-class field \(\mathrm{F}_p^2(K)\) of \(K\),
is a metabelian \(p\)-group
with abelianisation \(G/\gamma_2(G)\) of type \((p,p)\)
\cite{Ma}.
In the exceptional case \(p=2\) it is always of maximal class.

By the computation of the transfer kernels \(\mathrm{Ker}(\mathrm{V}_i)\) of \(G\)
we determine the principalisation type of the
\(p\)-class group \(\mathrm{Cl}_p(K)\) of \(K\)
in the \(p+1\) unramified cyclic extensions \(N_i\vert K\)
of relative degree \(p\) in subsection \ref{ss:MaxKerTrf}.
For \(p\ge 5\), our results are very similar to those of
B. Nebelung \cite[p.202, Satz 6.9]{Ne} for \(p=3\).
However, in the exceptional case \(p=2\)
our precise principalisation types
permit more explicit statements than H. Kisilevsky \cite[p.273, Th.2]{Ki2}
or E. Benjamin and C. Snyder \cite[p.163, \S 2]{BeSn},
since the cohomology
\(\mathrm{H}^0(\mathrm{Gal}(N_i\vert K),\mathrm{Cl}_p(N_i))\)
of the \(p\)-class groups \(\mathrm{Cl}_p(N_i)\)
with respect to the relative Galois groups \(\mathrm{Gal}(N_i\vert K)\) \cite{Ki1}
only yields the coarse distinction
of the conditions (A) and (B) in the sense of O. Taussky \cite[p.435]{Ta}.

Transfer images and kernels of
metabelian \(3\)-groups \(G\)
with abelianisation \(G/\gamma_2(G)\) of type \((3,3)\)
and of coclass \(\mathrm{cc}(G)\ge 2\)
are investigated in section \ref{s:LowTrf},
based on presentations given by Nebelung \cite{Ne}.

In subsection \ref{ss:LowKerTrf}
we show for certain \(3\)-groups \(G\) of coclass \(\mathrm{cc}(G)=r\ge 3\)
that it can be decided
by means of the parity of the index of nilpotency of the second \(3\)-class group \(G\)
of a quadratic number field \(K=\mathbb{Q}(\sqrt{D})\),
if \(G\) is represented by a terminal or an internal node
on the finitely many directed and rooted trees of isomorphism classes
of metabelian \(3\)-groups of coclass \(r\) with abelianisation of type \((3,3)\)
in the sense of Nebelung \cite[p.181 ff]{Ne},
which are infinite subgraphs of the coclass graph \(\mathcal{G}(3,r)\)
in the sense of Ascione, Leedham-Green, et al. \cite{As,LgMk,EiLg,DEF}.

At the end of sections \ref{s:MaxTrf} and \ref{s:LowTrf}
we present a complete overview of all transfer types
of metabelian \(p\)-groups
with abelianisation of type \((p,p)\)
for \(p=2\) and \(p=3\).


\section{Transfers of a metabelian \(p\)-group of maximal class}
\label{s:MaxTrf}

With an arbitrary prime \(p\ge 2\),
let \(G\) be a metabelian \(p\)-group
of order \(\lvert G\rvert=p^m\) and
nilpotency class \(\mathrm{cl}(G)=m-1\), where \(m\ge 3\).
Then \(G\) is of maximal class
and the commutator factor group \(G/\gamma_2(G)\) of \(G\) is of type \((p,p)\)
\cite{Bl,Mi}.
The lower central series of \(G\) is defined
recursively by \(\gamma_1(G)=G\) and
\(\gamma_j(G)=\lbrack\gamma_{j-1}(G),G\rbrack\) for \(j\ge 2\).

The centraliser
\(\chi_2(G)
=\lbrace g\in G\mid\lbrack g,u\rbrack\in\gamma_4(G)\text{ for all }u\in\gamma_2(G)\rbrace\)
of the two-step factor group \(\gamma_2(G)/\gamma_4(G)\), that is,
\[\chi_2(G)/\gamma_4(G)
=\mathrm{Centraliser}_{G/\gamma_4(G)}(\gamma_2(G)/\gamma_4(G))\,,\]
is the biggest subgroup of \(G\) such that
\(\lbrack\chi_2(G),\gamma_2(G)\rbrack\le\gamma_4(G)\).
It is characteristic, contains the commutator group \(\gamma_2(G)\), and
coincides with \(G\), if and only if \(m=3\).
Let the isomorphism invariant \(k=k(G)\) of \(G\) be defined by
\[\lbrack\chi_2(G),\gamma_2(G)\rbrack=\gamma_{m-k}(G)\,,\]
where \(k=0\) for \(m=3\), \(0\le k\le m-4\) for \(m\ge 4\),
and \(0\le k\le\min\lbrace m-4,p-2\rbrace\) for \(m\ge p+1\),
according to Miech \cite[p.331]{Mi}.

Suppose that generators of \(G=\langle x,y\rangle\) are selected such that
\(x\in G\setminus\chi_2(G)\), if \(m\ge 4\), and \(y\in\chi_2(G)\setminus\gamma_2(G)\).
We define the main commutator
\(s_2=\lbrack y,x\rbrack\in\gamma_2(G)\)
and the higher commutators
\(s_j=\lbrack s_{j-1},x\rbrack=s_{j-1}^{x-1}\in\gamma_j(G)\) for \(j\ge 3\).
In the sequel,
we only need two relations for \(p\)th powers of the generators \(x\) and \(y\) of \(G\),
when we calculate explicit images of the transfers of \(G\),

\begin{equation}
\label{eqn:MaxRel}
x^p=s_{m-1}^w\quad\text{ and }\quad
y^p\prod_{\ell=2}^p\,s_\ell^{\binom{p}{\ell}}=s_{m-1}^z
\quad \text{ with exponents }\quad 0\le w,z\le p-1\,,
\end{equation}

\noindent
according to Miech \cite[p.332, Th.2, (3)]{Mi}.
Blackburn uses the notation \(\delta=w\) and \(\gamma=z\)
for these relational exponents \cite[p.84, (36),(37)]{Bl}.

Additionally, the group \(G\) satisfies
relations for \(p\)th powers of the higher commutators,
\[s_{j+1}^p\prod_{\ell=2}^p\,s_{j+\ell}^{\binom{p}{\ell}}=1\quad\text{ for }1\le j\le m-2\,,\]
and the main commutator relation of Miech \cite[p.332, Th.2, (2)]{Mi},

\begin{equation}
\label{eqn:MaxComRel}
\lbrack y,s_2\rbrack=\prod_{r=1}^k s_{m-r}^{a(m-r)}
\in\lbrack\chi_2(G),\gamma_2(G)\rbrack=\gamma_{m-k}(G)\,,
\end{equation}

\noindent
with exponents \(0\le a(m-r)\le p-1\) for \(1\le r\le k\), \(a(m-k)>0\).
Blackburn restricts his investigations to \(k\le 2\) and uses the notation
\(\beta=a(m-1)\) and \(\alpha=a(m-2)\) \cite[p.82, (33)]{Bl}.

By \(G_a^{(m)}(z,w)\) we denote 
the representative of an isomorphism class of
metabelian \(p\)-groups \(G\) of maximal class
and of order \(\lvert G\rvert=p^m\),
which satisfies the relations
(\ref{eqn:MaxComRel}) and (\ref{eqn:MaxRel})
with a fixed system of exponents
\(a=(a(m-k),\ldots,a(m-1))\) and \(w\) and \(z\).

The maximal normal subgroups \(M_i\) of \(G\)
contain the commutator group \(\gamma_2(G)\) of \(G\)
as a normal subgroup of index \(p\) and thus
are of the shape \(M_i=\langle g_i,\gamma_2(G)\rangle\).
We define an order by
\(g_1=y\) and \(g_i=xy^{i-2}\) for \(2\le i\le p+1\).
The commutator groups \(\gamma_2(M_i)\)
are of the general form
\(\gamma_2(M_i)=\langle s_2,\ldots,s_{m-1}\rangle^{g_i-1}\),
according to \cite[Cor.3.1.1]{Ma},
and in particular

\begin{equation}
\label{eqn:MaxComGrp}
\begin{array}{rcl}
\gamma_2(M_1)&=&
\begin{cases}
1,&\text{ if }k=0,\\
\gamma_{m-k}(G),&\text{ if }k\ge 1,
\end{cases}\\
\gamma_2(M_i)&=&\gamma_3(G)\quad\text{ for }2\le i\le p+1.
\end{array}
\end{equation}

For an arbitrary fixed index \(1\le i\le p+1\)
we select an element \(h\in G\setminus M_i\)
and denote the \(p\)th \textit{trace element} (\textit{Spur}) of \(h\)
in the group ring \(\mathbb{Z}\lbrack G\rbrack\) by
\(\mathrm{S}_p(h)=\sum_{\ell=1}^p\,h^{\ell-1}\).

Taking into consideration that \(G^p<\gamma_2(G)<M_i\),
the \textit{transfer} (\textit{Verlagerung}) \(\mathrm{V}_i=\mathrm{V}_{G,M_i}\)
from \(G\) to \(M_i\) is given by

\begin{equation}
\label{eqn:MaxTrf}
\mathrm{V}_i:G/\gamma_2(G)\longrightarrow M_i/\gamma_2(M_i),
\ g\gamma_2(G)\mapsto
\begin{cases}
g^p\gamma_2(M_i),&\text{ if }g\in G\setminus M_i,\\
g^{\mathrm{S}_p(h)}\gamma_2(M_i),&\text{ if }g\in M_i,
\end{cases}
\end{equation}

\noindent
according to E. Artin \cite[p.50]{Ar2},
H. Hasse \cite[p.171, VII]{Ha2},
or K. Miyake \cite[p.296 ff]{My}.


\subsection{Images of the transfers}
\label{ss:MaxTrmTrf}

Since the elementary abelian bicyclic \(p\)-group \(G\) of type \((p,p)\)
represents a degenerate case of the metabelian \(p\)-groups of maximal class,
which cannot be excluded generally
in the subsequent number theoretical applications,
we begin by determining its transfers.

\begin{theorem}
\label{t:AblTrmTrf}

With an arbitrary prime \(p\ge 2\),
let \(G\) be the elementary abelian bicyclic \(p\)-group of type \((p,p)\)
and order \(\lvert G\rvert=p^m\), \(m=2\)
with its \(p+1\) cyclic subgroups
as maximal normal subgroups \(M_1,\ldots,M_{p+1}\).
Let
\[\mathrm{V}_i:G/\gamma_2(G)\longrightarrow M_i/\gamma_2(M_i),
\ g\gamma_2(G)\mapsto\mathrm{V}_i(g\gamma_2(G))\]
be the transfer from \(G\) to \(M_i\) for \(1\le i\le p+1\).

Then the images of all transfers are trivial,
\[\mathrm{V}_i(g\gamma_2(G))=1\quad\text{ for }g\in G\text{ and }1\le i\le p+1\,.\]

\end{theorem}

\begin{proof}

Like any elementary abelian \(p\)-group, \(G\) is of exponent \(p\).
Further, all commutator groups \(\gamma_2(M_i)=1\) are trivial.
In formula (\ref{eqn:MaxTrf}),
with an arbitrary element \(h\in G\setminus M_i\),
we consequently obtain
the triviality of the explicit image of each transfer,
\[\mathrm{V}_i(g\gamma_2(G))=
\begin{cases}
g^p\gamma_2(M_i)=1,&\text{ if }g\in G\setminus M_i,\\
g^{\mathrm{S}_p(h)}\gamma_2(M_i)=g^{\sum_{\ell=1}^p\,h^{\ell-1}}
=\prod_{\ell=1}^p\,g^{h^{\ell-1}}=\prod_{\ell=1}^p\,g=g^p=1,&\text{ if }g\in M_i,
\end{cases}\]
for \(1\le i\le p+1\).
\end{proof}

After this degenerate case we treat the non-abelian \(p\)-groups \(G\)
of maximal class with abelian commutator group \(\gamma_2(G)\)
and commutator factor group \(G/\gamma_2(G)\) of type \((p,p)\)
for the uniform standard case of odd primes \(p\ge 3\).

\begin{theorem}
\label{t:MaxTrmTrf}

With an odd prime \(p\ge 3\),
let \(G=\langle x,y\rangle\) be a metabelian \(p\)-group of maximal class
of order \(\lvert G\rvert=p^m\), where \(m\ge 3\).
Suppose that the generators of \(G\) are selected such that
\(x\in G\setminus\chi_2(G)\), if \(m\ge 4\), \(y\in\chi_2(G)\setminus\gamma_2(G)\),
and the relations (\ref{eqn:MaxRel})
with exponents \(0\le w,z\le p-1\) are satisfied.
Let the maximal normal subgroups \(M_1,\ldots,M_{p+1}\) be ordered by
\(M_1=\langle y,\gamma_2(G)\rangle\) and
\(M_i=\langle xy^{i-2},\gamma_2(G)\rangle\) for \(2\le i\le p+1\).
Finally, let
\[\mathrm{V}_i:G/\gamma_2(G)\longrightarrow M_i/\gamma_2(M_i),
\ g\gamma_2(G)\mapsto\mathrm{V}_i(g\gamma_2(G))\]
denote the transfer from \(G\) to \(M_i\) for \(1\le i\le p+1\).

Assume that the cosets \(g\gamma_2(G)\in G/\gamma_2(G)\)
are represented in the shape
\(g\equiv x^jy^\ell\pmod{\gamma_2(G)}\) with \(0\le j,\ell\le p-1\),
then the images of the transfers are generally given by
\[\mathrm{V}_i(x^jy^\ell\gamma_2(G))=s_{m-1}^{wj+z\ell}\gamma_2(M_i)\quad\text{ for }1\le i\le p+1.\]
With the explicit form (\ref{eqn:MaxComGrp}) of the commutator groups \(\gamma_2(M_i)\), they are given by

\begin{eqnarray*}
\label{e:MaxTrfDtl}
\mathrm{V}_1(x^jy^\ell\gamma_2(G))&=&
\begin{cases}
s_{m-1}^{wj+z\ell}\cdot 1,&\text{ if }\lbrack\chi_2(G),\gamma_2(G)\rbrack=1,\ k=0,\ m\ge 3,\ p\ge 3,\\
1\cdot\gamma_{m-1}(G),&\text{ if }\lbrack\chi_2(G),\gamma_2(G)\rbrack=\gamma_{m-1}(G),\ k=1,\ m\ge 5,\ p\ge 3,\\
1\cdot\gamma_{m-k}(G),&\text{ if }\lbrack\chi_2(G),\gamma_2(G)\rbrack=\gamma_{m-k}(G),\ k\ge 2,\ m\ge 6,\ p\ge 5,
\end{cases}\\
\mathrm{V}_i(x^jy^\ell\gamma_2(G))&=&
\begin{cases}
s_2^{wj+z\ell}\cdot 1,&\text{ if }m=3,\\
1\cdot \gamma_3(G),&\text{ if }m\ge 4,
\end{cases}
\qquad\qquad\text{ for }2\le i\le p+1.
\end{eqnarray*}

\end{theorem}

\begin{proof}

By means of formula (\ref{eqn:MaxTrf}), we calculate
the explicit image \(\mathrm{V}_i(g\gamma_2(G))\) of the transfer,
if the element \(g\) is represented
by the generators \(x,y\) of \(G\) in the shape
\(g\equiv x^jy^\ell\pmod{\gamma_2(G)}\) with \(0\le j,\ell\le p-1\).

\begin{enumerate}

\item
We start with the first maximal normal subgroup
\(M_1=\langle y,\gamma_2(G)\rangle\), for which
\(x\in G\setminus M_1\) and \(y\in M_1\)
and thus\\
\(\mathrm{V}_1(x\gamma_2(G))=x^p\gamma_2(M_1)=s_{m-1}^w\gamma_2(M_1)\),\\
\(\mathrm{V}_1(y\gamma_2(G))=y^{\mathrm{S}_p(x)}\gamma_2(M_1)=y^{\sum_{j=1}^p\,x^{j-1}}\gamma_2(M_1)\).\\
We distinguish two cases, according to the invariant \(k\) of \(G\).\\
If \(k=0\), then \(\gamma_2(M_1)=1\),
\(M_1\) is an abelian normal subgroup,
and we can use symbolic powers with the difference \(X=x-1\in\mathbb{Z}\lbrack G\rbrack\) in the exponents:\\
\(\mathrm{V}_1(y\gamma_2(G))=y^{\sum_{j=1}^p\,\binom{p}{j}X^{j-1}}
=\prod_{j=1}^p\,(y^{X^{j-1}})^{\binom{p}{j}}=y^p\prod_{j=2}^p\,s_j^{\binom{p}{j}}=s_{m-1}^z\).\\
However, if \(k\ge 1\), then \(\gamma_2(M_1)=\gamma_{m-k}(G)\ge\gamma_{m-1}(G)>1\).
Here we must derive a special case of a commutator formula by Miech \cite[p.338, Lem.2]{Mi},\\
\(\lbrack y,x^j\rbrack=\lbrack y,x\rbrack^{\mathrm{S}_j(x)}=s_2^{\sum_{r=1}^j\,x^{r-1}}
=s_2^{\sum_{r=1}^j\,\binom{j}{r}X^{r-1}}=\prod_{r=1}^j\,(s_2^{X^{r-1}})^{\binom{j}{r}}
=\prod_{r=1}^j\,s_{r+1}^{\binom{j}{r}}\),\\
and we obtain \(\mathrm{V}_1(y\gamma_2(G))=\)\\
\(=y^{\sum_{j=0}^{p-1}\,x^j}\gamma_2(M_1)
=\prod_{j=0}^{p-1}\,y^{x^j}\gamma_2(M_1)=\prod_{j=0}^{p-1}\,(y\lbrack y,x^j\rbrack)\gamma_2(M_1)
=\prod_{j=0}^{p-1}\,(y\prod_{r=1}^j\,s_{r+1}^{\binom{j}{r}})\gamma_2(M_1)\).\\
Now we observe that \(y\) commutes with all elements \(s_\ell\) modulo \(\gamma_2(M_1)\),\\
\(s_\ell y=ys_\ell\lbrack s_\ell,y\rbrack\equiv ys_\ell\pmod{\gamma_2(M_1)}\),\\
since
\(\lbrack s_\ell,y\rbrack\in\langle s_2,\ldots,s_{m-1}\rangle^{y-1}=\gamma_2(G)^{y-1}=\gamma_2(M_1)\)
for all \(\ell\ge 2\), and we obtain\\
\(\mathrm{V}_1(y\gamma_2(G))
=(\prod_{j=0}^{p-1}\,y)(\prod_{j=0}^{p-1}\,\prod_{r=1}^j\,s_{r+1}^{\binom{j}{r}})\gamma_2(M_1)\).\\
An interchange of the order of the factors in the double product yields\\
\(\mathrm{V}_1(y\gamma_2(G))
=y^p\prod_{r=1}^{p-1}\,\prod_{j=r}^{p-1}\,s_{r+1}^{\binom{j}{r}}\gamma_2(M_1)\)\\
\(=y^p\prod_{t=2}^p\,\prod_{j=t-1}^{p-1}\,s_t^{\binom{j}{t-1}}\gamma_2(M_1)
=y^p\prod_{t=2}^p\,s_t^{\sum_{j=t-1}^{p-1}\,{\binom{j}{t-1}}}\gamma_2(M_1)\).\\
Using the index transformation \(s=t-1\), \(\ell=j-s\), \(n=p-1-s\)
and the following well-known identity for binomial coefficients,
\[\sum_{\ell=0}^n\,\binom{s+\ell}{s}=\binom{s+n+1}{s+1}\text{ with }s,n\ge 0,\]
we finally obtain\\
\(\mathrm{V}_1(y\gamma_2(G))=y^p\prod_{t=2}^p\,s_t^{\binom{p}{t}}\gamma_2(M_1)
=s_{m-1}^z\gamma_2(M_1)\).

\item
The second maximal normal subgroup is
\(M_2=\langle x,\gamma_2(G)\rangle\), whence
\(x\in M_2\) and \(y\in G\setminus M_2\).
Therefore,\\
\(\mathrm{V}_2(x\gamma_2(G))=x^{\mathrm{S}_p(y)}\gamma_2(M_2)=x^{\sum_{j=0}^{p-1}\,y^j}\gamma_3(G)
=\prod_{j=0}^{p-1}\,x^{y^j}\gamma_3(G)=\prod_{j=0}^{p-1}\,(x\lbrack x,y^j\rbrack)\gamma_3(G)\).\\
Now we use a congruence modulo \(\gamma_3(G)\) for commutators of powers,
\[\lbrack x,y^j\rbrack\equiv\lbrack x,y\rbrack^j\pmod{\gamma_3(G)},\]
which is due to Blackburn \cite[p.49, Th.1.4]{Bl},
and we observe that
\[\lbrack x,y\rbrack x
=x\lbrack x,y\rbrack\cdot\lbrack\lbrack x,y\rbrack ,x\rbrack\equiv x\lbrack x,y\rbrack\pmod{\gamma_3(G)}.\]
\(\mathrm{V}_2(x\gamma_2(G))=\prod_{j=0}^{p-1}\,(x\lbrack x,y\rbrack^j)\gamma_3(G)
=(\prod_{j=0}^{p-1}\,x)(\prod_{j=0}^{p-1}\,\lbrack x,y\rbrack^j)\gamma_3(G)
=x^p(s_2^{-1})^{\sum_{j=0}^{p-1}\,j}\gamma_3(G)=x^ps_2^{-\binom{p}{2}}\gamma_3(G)
=x^p\gamma_3(G)=s_{m-1}^w\gamma_3(G)\),
since \((\gamma_2(G):\gamma_3(G))=p\) and thus \(s_2^p\in\gamma_3(G)\).
Here it is essential that \(p\) divides the binomial coefficient \(\binom{p}{2}\),
which is only correct for odd \(p\ge 3\).
Further, we have\\
\(\mathrm{V}_2(y\gamma_2(G))=y^p\gamma_2(M_2)
=y^p\prod_{j=2}^p\,s_j^{\binom{p}{j}}\gamma_3(G)
=s_{m-1}^z\gamma_3(G)\),
since \(s_j\in\gamma_3(G)\) for \(j\ge 3\)
and \(s_2^{\binom{p}{2}}\in\gamma_3(G)\), provided that \(p\ge 3\).

\item
All the other maximal normal subgroups \(M_i\) with \(3\le i\le p+1\)
can be treated in a uniform way.
Since \(M_i=\langle xy^{i-2},\gamma_2(G)\rangle\), we have
\(x\in G\setminus M_i\) und \(y\in G\setminus M_i\)
and thus\\
\(\mathrm{V}_i(x\gamma_2(G))=x^p\gamma_2(M_i)=s_{m-1}^w\gamma_3(G)\), similarly as for \(i=1\),\\
\(\mathrm{V}_i(y\gamma_2(G))=y^p\gamma_2(M_i)=s_{m-1}^z\gamma_3(G)\), similarly as for \(i=2\).
\end{enumerate}

\noindent
According to \cite[p.82]{Bl}, the possible maximum of the invariant \(k\)
depends on \(m\) and \(p\).
\end{proof}

Now we separately analyse the exceptional case \(p=2\) with its irregular expressions for the transfers.

\begin{theorem}
\label{t:TwoTrmTrf}

Let \(G=\langle x,y\rangle\) be a metabelian \(2\)-group of maximal class
and of order \(\lvert G\rvert=2^m\), where \(m\ge 3\).
Suppose that the generators of \(G\) are selected such that
\(x\in G\setminus\chi_2(G)\), if \(m\ge 4\), \(y\in\chi_2(G)\setminus\gamma_2(G)\),
and the relations (\ref{eqn:MaxRel})
are satisfied with exponents \(0\le w,z\le 1\).
Let the maximal normal subgroups \(M_1,\ldots,M_3\) be ordered by
\(M_1=\langle y,\gamma_2(G)\rangle\), \(M_2=\langle x,\gamma_2(G)\rangle\),
and \(M_3=\langle xy,\gamma_2(G)\rangle\).
Assume that
\[\mathrm{V}_i:G/\gamma_2(G)\longrightarrow M_i/\gamma_2(M_i),
\ g\gamma_2(G)\mapsto\mathrm{V}_i(g\gamma_2(G))\]
denotes the transfer from \(G\) to \(M_i\) for \(1\le i\le 3\).

If the cosets \(g\gamma_2(G)\in G/\gamma_2(G)\) are represented by
\(g\equiv x^jy^\ell\pmod{\gamma_2(G)}\) with \(0\le j,\ell\le 1\),
then the images of the transfers \(\mathrm{V}_i\)
reveal an irregular behavior for \(2\le i\le 3\)
and are given by

\begin{eqnarray*}
\label{e:TwoTrmTrf}
\mathrm{V}_1(x^jy^\ell\gamma_2(G))&=&s_{m-1}^{wj+z\ell}\cdot 1,\\
\mathrm{V}_2(x^jy^\ell\gamma_2(G))&=&s_{m-1}^{wj+z\ell}s_2^{-(j+\ell)}\gamma_3(G)=
\begin{cases}
s_2^{(w-1)j+(z-1)\ell}\cdot 1,&\text{ if }m=3,\\
s_2^{-(j+\ell)}\gamma_3(G),&\text{ if }m\ge 4,
\end{cases}\\
\mathrm{V}_3(x^jy^\ell\gamma_2(G))&=&s_{m-1}^{wj+z\ell}s_2^{-\ell}\gamma_3(G)=
\begin{cases}
s_2^{wj+(z-1)\ell}\cdot 1,&\text{ if }m=3,\\
s_2^{-\ell}\gamma_3(G),&\text{ if }m\ge 4.
\end{cases}
\end{eqnarray*}

\end{theorem}

\begin{proof}

According to \cite[p.23, Th.1.2]{Bv} and \cite[Ch.5, Th.4.5]{Go},
\(G\) is isomorphic to
either a dihedral group of order \(2^m\), \(m\ge 3\),
or a generalised quaternion group of order \(2^m\), \(m\ge 3\),
or a semi-dihedral group of order \(2^m\), \(m\ge 4\).
In each case, we have \(s_j=y^{(-2)^{j-1}}\) for \(j\ge 3\).
Thus \(\gamma_2(G)=\langle s_2,\ldots,s_{m-1}\rangle=\langle s_2\rangle=\langle y^2\rangle\),
and the representation of the three maximal normal subgroups \(M_i\) of \(G\) becomes simply

\begin{eqnarray*}
\label{e:TwoNrm}
M_1&=&\langle y,\gamma_2(G)\rangle=\langle y,s_2\rangle=\langle y\rangle\text{ (cyclic)},\\
M_2&=&\langle x,\gamma_2(G)\rangle=\langle x,s_2\rangle,\\
M_3&=&\langle xy,\gamma_2(G)\rangle=\langle xy,s_2\rangle.\\
\end{eqnarray*}

The commutator groups \(\gamma_2(M_i)\) are of the form

\begin{eqnarray*}
\label{e:TwoComGrp}
\gamma_2(M_1)&=&1,\text{ since }k=0,\\
\gamma_2(M_i)&=&\gamma_3(G)\text{ for }2\le i\le 3.
\end{eqnarray*}

Similarly as in the proof of theorem \ref{t:MaxTrmTrf}, we select
an element \(h\in G\setminus M_i\) and denote by
\(\mathrm{S}_2(h)=\sum_{j=1}^2\,h^{j-1}=1+h\in\mathbb{Z}\lbrack G\rbrack\)
the second trace element of \(h\), for an arbitrary fixed index \(1\le i\le 3\).
Then the transfer from \(G\) to \(M_i\) is given by
\[\mathrm{V}_i:G/\gamma_2(G)\longrightarrow M_i/\gamma_2(M_i),
\ g\gamma_2(G)\mapsto
\begin{cases}
g^2\gamma_2(M_i),&\text{ if }g\in G\setminus M_i,\\
g^{1+h}\gamma_2(M_i),&\text{ if }g\in M_i,
\end{cases}\]
according to formula (\ref{eqn:MaxTrf}).
By means of these formulas we calculate
the explicit image \(\mathrm{V}_i(g\gamma_2(G))\) of the transfer,
where the element \(g\) is represented
by the generators \(x,y\) of \(G\) in the shape
\(g\equiv x^jy^\ell\pmod{\gamma_2(G)}\) with \(0\le j,\ell\le 1\).

We partially obtain irregular images of the transfers,
since the main commutator \(s_2\) appears only in its first power
in the Blackburn relation \(y^2s_2=s_{m-1}^z\) for the second power of \(y\),
and \(s_2\) is not contained in \(\gamma_3(G)\).
Irregularities only occur for \(\mathrm{V}_2\) and \(\mathrm{V}_3\),
but not for the abelian maximal normal subgroup \(M_1\).
The images of the generators are

\begin{eqnarray*}
\label{e:TwoTrf1}
\mathrm{V}_1(y)&=&y^{1+x}\gamma_2(M_1)=yx^{-1}yx\cdot 1=y^2y^{-1}x^{-1}yx\cdot 1=y^2s_2\cdot 1=s_{m-1}^z\cdot 1,\\
\mathrm{V}_1(x)&=&x^2\gamma_2(M_1)=s_{m-1}^w\cdot 1,\\
\mathrm{V}_2(y)&=&y^2\gamma_2(M_2)=s_{m-1}^z s_2^{-1}\gamma_3(G)\text{ (irregular)},\\
\mathrm{V}_2(x)&=&x^{1+y}\gamma_2(M_2)=xy^{-1}xy\gamma_3(G)=x^2x^{-1}y^{-1}xy\gamma_3(G)=x^2s_2^{-1}\gamma_3(G)
=s_{m-1}^w s_2^{-1}\gamma_3(G)\text{ (irregular)},\\
\mathrm{V}_3(y)&=&y^2\gamma_2(M_3)=s_{m-1}^z s_2^{-1}\gamma_3(G)\text{ (irregular)},\\
\mathrm{V}_3(x)&=&x^2\gamma_2(M_3)=s_{m-1}^w\gamma_3(G).
\end{eqnarray*}
\end{proof}


\subsection{Singulets and multiplets of transfer types}
\label{ss:TypTrf}

For the kernel of the transfer
\[\mathrm{V}_i:G/\gamma_2(G)\longrightarrow M_i/\gamma_2(M_i),
\ g\gamma_2(G)\mapsto\mathrm{V}_i(g\gamma_2(G))\]
from a \(p\)-group \(G\)
with abelian commutator group \(\gamma_2(G)\)
and commutator factor group \(G/\gamma_2(G)\) of type \((p,p)\)
to one of its maximal normal subgroups \(M_i\), where \(1\le i\le p+1\),
we have \(p+3\) possibilities,
\[\text{either }\mathrm{Ker}(\mathrm{V}_i)=1\quad\text{ or }
\mathrm{Ker}(\mathrm{V}_i)=M_j/\gamma_2(G)\text{ with }1\le j\le p+1
\quad\text{ or }\mathrm{Ker}(\mathrm{V}_i)=G/\gamma_2(G)\,.\]
However, in subsection \ref{ss:MaxKerTrf} it will turn out that
a transfer \(\mathrm{V}_i\) is never injective, that is,
the transfer kernel \(\mathrm{Ker}(\mathrm{V}_i)>1\) is always non-trivial,
provided that \(G\) is of maximal class.

Consequently, there remain \(p+2\) possible
\textit{singulets of transfer types} \(\varkappa(i)\)
for each individual maximal normal subgroup \(M_i<G\):

\begin{itemize}

\item
either a \textit{total} transfer with two-dimensional kernel,
\[\mathrm{Ker}(\mathrm{V}_i)=G/\gamma_2(G),\text{ denoted by the symbol }\varkappa(i)=0\,,\]

\item
or a \textit{partial} transfer with one-dimensional kernel,
\[\mathrm{Ker}(\mathrm{V}_i)=M_j/\gamma_2(G)
\text{ for some }1\le j\le p+1,\text{ denoted by the symbol }\varkappa(i)=j\,.\]

\end{itemize}

\noindent
According to Taussky \cite[p.435]{Ta}, we also have a coarser distinction:

\begin{itemize}

\item
condition (A),
\[\mathrm{Ker}(\mathrm{V}_i)\cap M_i/\gamma_2(G)>1\,,\]
since either \(\mathrm{Ker}(\mathrm{V}_i)=G/\gamma_2(G)\) with \(\varkappa(i)=0\)
for a total transfer
or \(\mathrm{Ker}(\mathrm{V}_i)=M_i/\gamma_2(G)\) with \(\varkappa(i)=i\)
for a partial transfer with \textit{fixed point} of \(\varkappa\),

\item
condition (B),
\[\mathrm{Ker}(\mathrm{V}_i)\cap M_i/\gamma_2(G)=1\,,\]
since \(\mathrm{Ker}(\mathrm{V}_i)=M_j/\gamma_2(G)\) with \(\varkappa(i)=j\ne i\)
for a partial transfer without fixed point,
where the exact value of \(j\) remains unknown.

\end{itemize}

\noindent
The singulets are combined to a
\textit{multiplet of transfer types}
for the family of all \(p+1\) maximal normal subgroups \(M_i<G\) with \(1\le i\le p+1\),
\[\varkappa=(\varkappa(1),\ldots,\varkappa(p+1))\in [0,p+1]^{p+1}\,.\]

\noindent
The number of total transfers
is an invariant \(\nu=\nu(G)\) of the \(p\)-group \(G\):

\begin{definition}
\label{d:TotTrf}
Let \(0\le\nu\le p+1\) be the number
\(\nu=\#\lbrace 1\le i\le p+1\mid\mathrm{Ker}(\mathrm{V}_i)=G/\gamma_2(G)\rbrace\)
of maximal normal subgroups \(M_i\) of \(G\),
for which the transfer \(\mathrm{V}_i\) from \(G\) to \(M_i\) is total.
\end{definition}

\noindent
We call a
multiplet \(\varkappa=(\varkappa(1),\ldots,\varkappa(p+1))\) of transfer types
\textit{partial}, if \(\varkappa(i)\ne 0\) for all \(1\le i\le p+1\),
that is, if \(\nu=0\),
and otherwise \textit{total}.

\noindent
The orbit \(\varkappa^{S_{p+1}}\) of a multiplet \(\varkappa\in [0,p+1]^{p+1}\)
under the operation \(\varkappa^\pi=\pi_0^{-1}\circ\varkappa\circ\pi\)
of the symmetric group \(S_{p+1}\) of degree \(p+1\)
is an invariant of the \(p\)-group \(G\),
which is independent from the order of the normal subgroups \(M_i\),
and is called the \textit{transfer type} \(\varkappa(G)\) of \(G\).
Here we assume that the extension \(\pi_0\) of \(\pi\in S_{p+1}\)
to \([0,p+1]\) fixes the zero.


\subsection{Extension and principalisation of ideal classes}
\label{ss:ExtPrc}

At the present position it is adequate
to turn to the number theoretical applications
of our purely group theoretical results.

For an arbitrary prime \(p\ge 2\),
the second \(p\)-class group \cite{Ma},
that is the Galois group \(G=\mathrm{Gal}(\mathrm{F}_p^2(K)\vert K)\)
of the second Hilbert \(p\)-class field \(\mathrm{F}_p^2(K)\),
of an algebraic number field \(K\)
with \(p\)-class group \(\mathrm{Cl}_p(K)\) of type \((p,p)\)
is a metabelian \(p\)-group \(G\)
with abelianisation \(G/\gamma_2(G)\) of type \((p,p)\).
The reason for this fact is that,
according to the reciprocity law of Artin \cite{Ar1}
and Galois correspondence, the commutator group
\[\gamma_2(G)=\mathrm{Gal}(\mathrm{F}_p^2(K)\vert\mathrm{F}_p^1(K))\simeq\mathrm{Cl}_p(\mathrm{F}_p^1(K))\]
is abelian as the \(p\)-class group of the first Hilbert \(p\)-class field \(\mathrm{F}_p^1(K)\) of \(K\)
and the commutator factor group
\[G/\gamma_2(G)
=\mathrm{Gal}(\mathrm{F}_p^2(K)\vert K)/\mathrm{Gal}(\mathrm{F}_p^2(K)\vert\mathrm{F}_p^1(K))
\simeq\mathrm{Gal}(\mathrm{F}_p^1(K)\vert K)
\simeq\mathrm{Cl}_p(K)\]
is of type \((p,p)\).

By the Galois correspondence \(M_i=\mathrm{Gal}(\mathrm{F}_p^2(K)\vert N_i)\),
the maximal normal subgroups \(M_1,\ldots,M_{p+1}\) of \(G\) are associated with
the \(p+1\) unramified cyclic extensions \(N_1,\ldots,N_{p+1}\)
of \(K\) of relative degree \(p\),
which are represented by the norm class groups
\(\mathrm{Norm}_{N_i\vert K}(\mathrm{Cl}_p(N_i))\) as subgroups of index \(p\)
in the \(p\)-class group \(\mathrm{Cl}_p(K)\) of \(K\),
according to \cite{Ar1}.
The abelianisations of the \(M_i\),
\[M_i/\gamma_2(M_i)
=\mathrm{Gal}(\mathrm{F}_p^2(K)\vert N_i)/\mathrm{Gal}(\mathrm{F}_p^2(K)\vert\mathrm{F}_p^1(N_i))
\simeq\mathrm{Gal}(\mathrm{F}_p^1(N_i)\vert N_i)
\simeq\mathrm{Cl}_p(N_i)\,,\]
are isomorphic to the \(p\)-class groups of the \(N_i\),
by \cite{Ar1}.

In algebraic number theory,
we are interested in the type of the principalisation
of ideal classes of \(\mathrm{Cl}_p(K)\)
in the \(p\)-class groups \(\mathrm{Cl}_p(N_i)\) \cite[5.3]{Ma}.
Therefore, we investigate
the kernel of the \textit{class extension} homomorphisms
\[\mathrm{j}_{N_i\vert K}:\mathrm{Cl}_p(K)\longrightarrow\mathrm{Cl}_p(N_i),\
\mathfrak{a}\mathcal{P}_K\mapsto(\mathfrak{a}\mathcal{O}_{N_i})\mathcal{P}_{N_i}\,,\]
where \(\mathcal{P}\) denotes the group of principal ideals
and \(\mathcal{O}\) the maximal order.
According to Hilbert's theorem \(94\) \cite[p.279]{Hi}, the class extension
\(\mathrm{j}_{N_i\vert K}\) has a non-trivial kernel
\(\mathrm{Ker}(\mathrm{j}_{N_i\vert K})>1\).
We define the \textit{multiplet} \(\varkappa\in [0,p+1]^{p+1}\)
\textit{of principalisation types} of \(K\) for \(1\le i\le p+1\) by
\[\mathrm{Ker}(\mathrm{j}_{N_i\vert K})=\mathrm{Norm}_{N_{\varkappa(i)}\vert K}(\mathrm{Cl}_p(N_{\varkappa(i)})),\]
if \(1\le\varkappa(i)\le p+1\), that is, for \textit{partial} principalisation,
and we put \(\varkappa(i)=0\) for \textit{total} principalisation,
\(\mathrm{Ker}(\mathrm{j}_{N_i\vert K})=\mathrm{Cl}_p(K)\).
The \textit{principalisation type} \(\varkappa(K)\) of \(K\) is defined
as the orbit \(\varkappa^{S_{p+1}}\) of the multiplet \(\varkappa\)
under the operation \(\varkappa^\pi=\pi_0^{-1}\circ\varkappa\circ\pi\)
of the symmetric group \(S_{p+1}\) of degree \(p+1\),
where we assume that the extension \(\pi_0\) of \(\pi\in S_{p+1}\)
to \([0,p+1]\) fixes the zero.\\
The extent of total principalisation
is expressed by an invariant \(\nu=\nu(K)\) of the field \(K\):

\begin{definition}
\label{d:TotPrc}
Let \(0\le\nu\le p+1\) be the number
\(\nu=\#\lbrace 1\le i\le p+1\mid\mathrm{Ker}(\mathrm{j}_{N_i\vert K})=\mathrm{Cl}_p(K)\rbrace\)
of unramified cyclic extension fields \(N_i\) of \(K\) of relative degree \(p\),
in which the entire \(p\)-class group \(\mathrm{Cl}_p(K)\) of \(K\)
becomes principal (cfr. \cite[capitulation  number, p.1230]{ChFt}.
\end{definition}

According to Artin \cite{Ar2} (see also Miyake \cite[p.297, Cor.]{My}),
the following commutative diagram establishes
the connection between the
number theoretical extensions \(\mathrm{j}_{N_i\vert K}\)
of \(p\)-class groups and the
group theoretical transfers \(\mathrm{V}_i=\mathrm{V}_{G,M_i}\)
from the abelianisation of the second \(p\)-class group
\(G=\mathrm{Gal}(\mathrm{F}_p^2(K)\vert K)\) of \(K\)
to the abelianisations of its maximal normal subgroups
\(M_i=\mathrm{Gal}(\mathrm{F}_p^2(K)\vert N_i)\) with \(1\le i\le p+1\).

\renewcommand{\arraystretch}{1.5}
\begin{table}[ht]
\label{tab:Transfer}
\begin{center}
\begin{tabular}{ccccc}
                  &                      & \(\mathrm{j}_{N_i\vert K}\) &                        &                  \\
                  & \(\mathrm{Cl}_p(K)\) & \(\longrightarrow\)         & \(\mathrm{Cl}_p(N_i)\) &                  \\
Artin isomorphism & \(\updownarrow\)     &                             & \(\updownarrow\)       & Artin isomorphism\\
                  & \(G/\gamma_2(G)\)    & \(\longrightarrow\)         & \(M_i/\gamma_2(M_i)\)  &                  \\
                  &                      & \(\mathrm{V}_{G,M_i}\)      &                        &                  \\
\end{tabular}
\end{center}
\end{table}

\noindent
Due to the commutativity of this diagram,
the following number theoretical and group theoretical concepts
correspond to each other:
principalisation kernels and transfer kernels,
\(\mathrm{Ker}(\mathrm{j}_{N_i\vert K})\simeq\mathrm{Ker}(\mathrm{V}_i)\) for \(1\le i\le p+1\),
norm class groups and cyclic subgroups of \(G/\gamma_2(G)\),
\(\mathrm{Norm}_{N_i\vert K}(\mathrm{Cl}_p(N_i))\simeq M_i/\gamma_2(G)\) for \(1\le i\le p+1\),
multiplets of principalisation types, multiplets of transfer types, and their orbits,
\(\varkappa(K)=\varkappa(G)\),
and finally the invariants
\(\nu(K)=\nu(G)\) describing the total principalisation and the total transfer.

\begin{remark}
\label{r:Hil94}
For an arbitrary prime \(p\ge 2\) and
an arbitrary base field \(K\) with \(p\)-class group of type \((p,p)\)
and second \(p\)-class group \(G=\mathrm{Gal}(\mathrm{F}_p^2(K)\vert K)\) of maximal class,
Hilbert's theorem \(94\) \cite{Hi},
that is the non-injectivity of the class extension homomorphisms \(\mathrm{j}_{N_i\vert K}\),
is a consequence of the non-injectivity of the transfers \(\mathrm{V}_i\), which
will be proved in the theorems \ref{t:AblKerTrf}, \ref{t:MaxKerTrf}, and \ref{t:TwoKerTrf}.
For second \(p\)-class groups \(G\) of non-maximal class
with \(G/\gamma_2(G)\) of type \((p,p)\),
which can occur only for \(p\ge 3\),
the non-injectivity of the transfers \(\mathrm{V}_i\)
has been proved by Nebelung for \(p=3\) \cite[p.208, Satz 6.14]{Ne}.
\end{remark}


\subsection{Kernels of the transfers}
\label{ss:MaxKerTrf}

By means of the expressions for the transfers,
which have been determined in subsection \ref{ss:MaxTrmTrf},
we are now able to calculate
the transfer kernels \(\mathrm{Ker}(\mathrm{V}_i)\) for \(1\le i\le p+1\)
and the transfer type \(\varkappa(G)\)
of the metabelian \(p\)-group \(G\) of maximal class.
We begin with the degenerate case of the
elementary abelian \(p\)-group of type \((p,p)\).

\begin{theorem}
\label{t:AblKerTrf}

Let \(p\ge 2\) be an arbitrary prime and
\(G\) the elementary abelian bicyclic \(p\)-group of type \((p,p)\)
and of order \(\lvert G\rvert=p^m\), \(m=2\)
with its \(p+1\) cyclic subgroups
as maximal normal subgroups \(M_1,\ldots,M_{p+1}\).
Then the following statements hold.

\begin{enumerate}

\item
The kernels of the transfers
\(\mathrm{V}_i:G/\gamma_2(G)\longrightarrow M_i/\gamma_2(M_i)\)
for \(1\le i\le p+1\)
are all two-dimensional of type \((p,p)\),
\[\mathrm{Ker}(\mathrm{V}_i)=G/\gamma_2(G)\,,\qquad\varkappa(i)=0\,.\]

\item
The multiplet \(\varkappa=(\varkappa(1),\ldots,\varkappa(p+1))\)
of the transfer types of \(G\) is given by
\[\varkappa=(\overbrace{0,\ldots,0}^{p+1\text{ times}}),\ \nu=p+1\,.\]

\end{enumerate}

\end{theorem}

\begin{proof}

By theorem \ref{t:AblTrmTrf}, all images of the transfers are trivial,
\(\mathrm{V}_i(g\gamma_2(G))=1\), for all elements \(g\in G\) and all \(1\le i\le p+1\).
Therefore we have \(\mathrm{Ker}(\mathrm{V}_i)=G/\gamma_2(G)\) for \(1\le i\le p+1\).
\end{proof}

Now we come to the uniform standard case
of a metabelian \(p\)-group of maximal class
with an odd prime \(p\ge 3\).

\begin{theorem}
\label{t:MaxKerTrf}

Let \(p\ge 3\) be an odd prime and
\(G\) a metabelian \(p\)-group of maximal class
of order \(\lvert G\rvert=p^n\)
and class \(\mathrm{cl}(G)=m-1\)
where \(n=m\ge 3\).
Suppose that generators \(x,y\) of \(G\) are selected such that
\(x\in G\setminus\chi_2(G)\), if \(m\ge 4\), \(y\in\chi_2(G)\setminus\gamma_2(G)\),
and the relations (\ref{eqn:MaxRel})
with exponents \(0\le w,z\le p-1\) are satisfied.
In the case of \(m=3\) and the extra special \(p\)-group
\(G\simeq G_0^{(3)}(0,1)\) of exponent \(p^2\),
let \(y\) be of order \(p\).
Let the maximal normal subgroups \(M_1,\ldots,M_{p+1}\) of \(G\) be ordered by
\(M_1=\langle y,\gamma_2(G)\rangle\) and
\(M_i=\langle xy^{i-2},\gamma_2(G)\rangle\) for \(2\le i\le p+1\).
Then, in particular \(M_1=\chi_2(G)\), if \(m\ge 4\),
and the following statements hold.

\begin{enumerate}

\item
The kernels of the transfers
\(\mathrm{V}_i:G/\gamma_2(G)\longrightarrow M_i/\gamma_2(M_i)\)
for \(1\le i\le p+1\)
are given by

\begin{eqnarray*}
\label{e:MaxKerDtl}
\mathrm{Ker}(\mathrm{V}_1)&=&
\begin{cases}
G/\gamma_2(G),&\text{ if }\lbrack\chi_2(G),\gamma_2(G)\rbrack=1,\ k=0,\ m\ge 3,\ w=0,\ z=0,\\
M_1/\gamma_2(G),&\text{ if }\lbrack\chi_2(G),\gamma_2(G)\rbrack=1,\ k=0,\ m\ge 3,\ w=1,\ z=0,\\
M_2/\gamma_2(G),&\text{ if }\lbrack\chi_2(G),\gamma_2(G)\rbrack=1,\ k=0,\ m\ge 4,\ w=0,\ z>0,\\
G/\gamma_2(G),&\text{ if }\lbrack\chi_2(G),\gamma_2(G)\rbrack=\gamma_{m-1}(G),\ k=1,\ m\ge 5,\ p\ge 3,\\
G/\gamma_2(G),&\text{ if }\lbrack\chi_2(G),\gamma_2(G)\rbrack=\gamma_{m-k}(G),\ k\ge 2,\ m\ge 6,\ p\ge 5,
\end{cases}\\
\mathrm{Ker}(\mathrm{V}_i)&=&
\begin{cases}
M_1/\gamma_2(G),&\text{ if }m=3,\ w=1,\ z=0,\\
G/\gamma_2(G),&\text{ if }m=3,\ w=0,\ z=0,\\
G/\gamma_2(G),&\text{ if }m\ge 4,
\end{cases}
\qquad\text{ for }2\le i\le p+1.
\end{eqnarray*}

\item
The singulets \(\varkappa(i)\) of transfer types of \(G\)
with \(1\le i\le p+1\) are given by

\begin{eqnarray*}
\label{e:MaxSng}
\varkappa(1)&=&
\begin{cases}
0,&\text{ if }m\ge 3,\ G\simeq G_0^{(m)}(0,0),\\
1,&\text{ if }m\ge 3,\ G=G_0^{(m)}(0,1),\\
2,&\text{ if }m\ge 4,\ G=G_0^{(m)}(z,0),\ z>0,\\
0,&\text{ if }\lbrack\chi_2(G),\gamma_2(G)\rbrack>1,\ k\ge 1,\ m\ge 5,
\end{cases}\\
\varkappa(i)&=&
\begin{cases}
1,&\text{ if }m=3,\ G=G_0^{(3)}(0,1),\\
0,&\text{ if }m=3,\ G\simeq G_0^{(3)}(0,0),\\
0,&\text{ if }m\ge 4,
\end{cases}
\qquad\text{ for }2\le i\le p+1.
\end{eqnarray*}

\item
The multiplet \(\varkappa=(\varkappa(1),\ldots,\varkappa(p+1))\)
of transfer types of \(G\) is given by

\begin{eqnarray*}
\label{e:MaxMlt}
\varkappa&=&
\begin{cases}
(\overbrace{1,\ldots,1}^{p+1\text{ times}}),\ \nu=0,&\text{ if }m=3,\ G=G_0^{(3)}(0,1),\\
(\overbrace{0,\ldots,0}^{p+1\text{ times}}),\ \nu=p+1,&\text{ if }m\ge 3,\ G\simeq G_0^{(m)}(0,0),\\
(1,\overbrace{0,\ldots,0}^{p\text{ times}}),\ \nu=p,&\text{ if }m\ge 4,\ G=G_0^{(m)}(0,1),\\
(2,\overbrace{0,\ldots,0}^{p\text{ times}}),\ \nu=p,&\text{ if }m\ge 4,\ G=G_0^{(m)}(z,0),\ z>0,\\
(\overbrace{0,\ldots,0}^{p+1\text{ times}}),\ \nu=p+1,&\text{ if }\lbrack\chi_2(G),\gamma_2(G)\rbrack>1,\ k\ge 1,\ m\ge 5.
\end{cases}
\end{eqnarray*}

\end{enumerate}

\end{theorem}

\begin{proof}

According to theorem \ref{t:MaxTrmTrf}, the images of the transfers
in dependence on the invariants \(k\) and \(m\), for elements
\(g\equiv x^jy^\ell\pmod{\gamma_2(G)}\) with \(0\le j,\ell\le p-1\),
are given by

\begin{eqnarray*}
\label{e:MaxTrfRcl}
\mathrm{V}_1(x^jy^\ell\gamma_2(G))&=&
\begin{cases}
s_{m-1}^{wj+z\ell}\cdot 1,&\text{ if }k=0,\\
1\cdot\gamma_{m-k}(G),&\text{ if }k\ge 1,
\end{cases}\\
\mathrm{V}_i(x^jy^\ell\gamma_2(G))&=&
\begin{cases}
s_2^{wj+z\ell}\cdot 1,&\text{ if }m=3,\\
1\cdot\gamma_3(G),&\text{ if }m\ge 4,
\end{cases}
\qquad\text{ for }2\le i\le p+1.
\end{eqnarray*}

\noindent
To determine the transfer kernel
\(\mathrm{Ker}(\mathrm{V}_i)\) for \(1\le i\le p+1\),
we have to solve the equation
\(\mathrm{V}_i(x^jy^\ell\gamma_2(G))=1\cdot\gamma_2(M_i)\)
with respect to the element \(x^jy^\ell\), that is, with respect to \(j,\ell\),
using the image of the transfer \(\mathrm{V}_i\).

Since we have the trivial image
\(\mathrm{V}_i(x^jy^\ell\gamma_2(G))
=1\cdot\gamma_3(G)=1\cdot\gamma_2(M_i)\)
for \(m\ge 4\) and \(2\le i\le p+1\),
the exponents \(0\le j,\ell\le p-1\) can both be selected arbitrarily,
and thus \(\mathrm{Ker}(\mathrm{V}_i)=G/\gamma_2(G)\)
is two-dimensional, that is \(\varkappa(i)=0\).\\
Similarly, we have the trivial image
\(\mathrm{V}_1(x^jy^\ell\gamma_2(G))
=1\cdot\gamma_{m-k}(G)=1\cdot\gamma_2(M_1)\)
for \(k\ge 1\) and \(i=1\),
and thus \(\mathrm{Ker}(\mathrm{V}_1)=G/\gamma_2(G)\), that is \(\varkappa(1)=0\).

In all other cases, we need the concrete relational exponents \(w,z\)
of the representative \(G_a^{(m)}(z,w)\) of the isomorphism class
of the given metabelian \(p\)-group \(G\) of maximal class,
where \(a\) denotes the system \((a(m-k),\ldots,a(m-1))\) of exponents
in the general relation
\[\lbrack y,s_2\rbrack=\prod_{r=1}^k s_{m-r}^{a(m-r)}
\in\lbrack\chi_2(G),\gamma_2(G)\rbrack=\gamma_{m-k}(G)\]
by Miech \cite[p.332, Th.2, (2)]{Mi},
which, however, does not enter the image of the transfer.

For \(m=3\), there are only the two isomorphism classes
of the extra special \(p\)-groups
\(G_0^{(3)}(0,0)\) of exponent \(p\) and
\(G_0^{(3)}(0,1)\) of exponent \(p^2\),
for which we have \(k=0\),
and thus \(a=0\) means the empty family \((a(m-r))_{1\le r\le k}\).\\
In the case \(w=0\), \(z=0\), we obtain the image
\(\mathrm{V}_i(x^jy^\ell\gamma_2(G))=s_2^{wj+z\ell}\cdot 1=1\)
for \(i=1\) and also for \(2\le i\le p+1\),
and thus \(\mathrm{Ker}(\mathrm{V}_i)=G/\gamma_2(G)\), that is \(\varkappa(i)=0\).\\
In the case \(w=1\), \(z=0\), however, we have the image
\(\mathrm{V}_i(x^jy^\ell\gamma_2(G))=s_2^{wj+z\ell}\cdot 1=s_2^j\cdot 1\)
for \(i=1\) and also for \(2\le i\le p+1\).
Therefore, we must have \(j=0\), whereas \(\ell\) can be selected arbitrarily.
Since \(\langle y,\gamma_2(G)\rangle=M_1\), it follows that
\(\mathrm{Ker}(\mathrm{V}_i)=M_1/\gamma_2(G)\), that is \(\varkappa(i)=1\).

It only remains to investigate \(i=1\) for \(m\ge 4\) and \(k=0\),
that is, for metabelian \(p\)-groups \(G\) of maximal class
with an abelian normal subgroup \(M_1\).
According to Blackburn \cite[p.88, Th.4.3]{Bl},
there exist two isomorphism classes
\(G_0^{(m)}(0,0)\), \(G_0^{(m)}(0,1)\) with \(z=0\),
and in general several isomorphism classes
\(G_0^{(m)}(z,0)\) with \(z>0\),
in this situation.\\
In the case \(w=0\), \(z=0\), we have
\(\mathrm{V}_1(x^jy^\ell\gamma_2(G))=s_{m-1}^{wj+z\ell}\cdot 1=1\)
and thus \(\mathrm{Ker}(\mathrm{V}_1)=G/\gamma_2(G)\), that is \(\varkappa(1)=0\).\\
In the case \(w=1\), \(z=0\), we obtain
\(\mathrm{V}_1(x^jy^\ell\gamma_2(G))=s_{m-1}^{wj+z\ell}\cdot 1=s_{m-1}^j\cdot 1\).
Thus, we must have \(j=0\), but \(\ell\) remains arbitrary.
Since \(\langle y,\gamma_2(G)\rangle=M_1\), it follows that
\(\mathrm{Ker}(\mathrm{V}_1)=M_1/\gamma_2(G)\), that is \(\varkappa(1)=1\).\\
In the case \(w=0\), \(z>0\), we have
\(\mathrm{V}_1(x^jy^\ell\gamma_2(G))=s_{m-1}^{wj+z\ell}\cdot 1=s_{m-1}^{z\ell}\cdot 1\).
This enforces \(\ell=0\), whereas \(j\) can be selected arbitrarily.
Since \(\langle x,\gamma_2(G)\rangle=M_2\), it follows that
\(\mathrm{Ker}(\mathrm{V}_1)=M_2/\gamma_2(G)\), that is \(\varkappa(1)=2\).

According to \cite[p.82]{Bl}, the possible maximum of the invariant \(k\) depends on \(m\) and \(p\).
\end{proof}

The irregular images of the transfers in the exceptional case \(p=2\)
exert a decisive influence on the transfer kernels
and cause considerable deviations of the multiplets of transfer types
from the uniform standard case of odd primes \(p\ge 3\).

\begin{theorem}
\label{t:TwoKerTrf}

Let \(G=\langle x,y\rangle\) be a metabelian \(2\)-group of maximal class
and of order \(\lvert G\rvert=2^m\), where \(m\ge 3\).
Suppose that the generators are selected such that
\(x\in G\setminus\chi_2(G)\), if \(m\ge 4\), \(y\in\chi_2(G)\setminus\gamma_2(G)\),
and the relations (\ref{eqn:MaxRel})
are satisfied with exponents \(0\le w,z\le 1\).
Let the maximal normal subgroups \(M_1,\ldots,M_3\) of \(G\) be ordered by
\(M_1=\langle y,\gamma_2(G)\rangle\),
\(M_2=\langle x,\gamma_2(G)\rangle\),
and \(M_3=\langle xy,\gamma_2(G)\rangle\).
Then \(M_1=\chi_2(G)\), if \(m\ge 4\),
and the following statements hold.

\begin{enumerate}

\item
The kernels of the transfers
\(\mathrm{V}_i:G/\gamma_2(G)\longrightarrow M_i/\gamma_2(M_i)\)
for \(1\le i\le 3\) are given by

\begin{eqnarray*}
\label{e:TwoKerDtl}
\mathrm{Ker}(\mathrm{V}_1)&=&
\begin{cases}
G/\gamma_2(G),&\text{ if }\ m\ge 3,\ w=0,\ z=0,\\
M_1/\gamma_2(G),&\text{ if }\ m\ge 3,\ w=1,\ z=0,\\
M_2/\gamma_2(G),&\text{ if }\ m\ge 4,\ w=0,\ z=1,\\
\end{cases}\\
\mathrm{Ker}(\mathrm{V}_2)&=&
\begin{cases}
M_2/\gamma_2(G),&\text{ if }m=3,\ w=1,\ z=0,\\
M_3/\gamma_2(G),&\text{ otherwise},
\end{cases}\\
\mathrm{Ker}(\mathrm{V}_3)&=&
\begin{cases}
M_3/\gamma_2(G),&\text{ if }m=3,\ w=1,\ z=0,\\
M_2/\gamma_2(G),&\text{ otherwise}.
\end{cases}
\end{eqnarray*}

\item
The singulets \(\varkappa(i)\) of transfer types of \(G\)
with \(1\le i\le 3\) are given by

\begin{eqnarray*}
\label{e:TwoSng}
\varkappa(1)&=&
\begin{cases}
0,&\text{ if }m\ge 3,\ G=G_0^{(m)}(0,0)=D(2^m),\\
1,&\text{ if }m\ge 3,\ G=G_0^{(m)}(0,1)=Q(2^m),\\
2,&\text{ if }m\ge 4,\ G=G_0^{(m)}(1,0)=S(2^m),
\end{cases}\\
\varkappa(2)&=&
\begin{cases}
2,&\text{ if }m=3,\ G=G_0^{(3)}(0,1)=Q(8),\\
3,&\text{ otherwise},
\end{cases}\\
\varkappa(3)&=&
\begin{cases}
3,&\text{ if }m=3,\ G=G_0^{(3)}(0,1)=Q(8),\\
2,&\text{ otherwise}.
\end{cases}
\end{eqnarray*}

\item
The multiplet \(\varkappa=(\varkappa(1),\varkappa(2),\varkappa(3))\)
of transfer types of \(G\) is given by

\begin{eqnarray*}
\label{e:TwoMlt}
\varkappa&=&
\begin{cases}
(1,2,3),\ \nu=0,&\text{ if }m=3,\ G=G_0^{(3)}(0,1)=Q(8),\\
(0,3,2),\ \nu=1,&\text{ if }m\ge 3,\ G=G_0^{(m)}(0,0)=D(2^m),\\
(1,3,2),\ \nu=0,&\text{ if }m\ge 4,\ G=G_0^{(m)}(0,1)=Q(2^m),\\
(2,3,2),\ \nu=0,&\text{ if }m\ge 4,\ G=G_0^{(m)}(1,0)=S(2^m).
\end{cases}
\end{eqnarray*}

\end{enumerate}

\end{theorem}

\begin{definition}
\label{d:CanMlt}
We call the multiplets of transfer types \(\varkappa\)
in the preceding theorems \ref{t:MaxKerTrf} and \ref{t:TwoKerTrf}
\textit{canonical},
since they only appear in this form,
if the group generators \(x,y\) are selected as described in the assumptions
and the maximal normal subgroups are arranged in the defined order.
\end{definition}

\begin{proof}

By theorem \ref{t:TwoTrmTrf},
the images of the transfers for elements of the shape
\(g\equiv x^jy^\ell\pmod{\gamma_2(G)}\), \(0\le j,\ell\le 1\),
and in dependence on the invariant \(m\),
are given by

\begin{eqnarray*}
\label{e:TwoTrfRcl}
\mathrm{V}_1(x^jy^\ell\gamma_2(G))&=&s_{m-1}^{wj+z\ell}\cdot 1,\\
\mathrm{V}_2(x^jy^\ell\gamma_2(G))&=&s_{m-1}^{wj+z\ell}s_2^{-(j+\ell)}\gamma_3(G)=
\begin{cases}
s_2^{(w-1)j+(z-1)\ell}\cdot 1,&\text{ if }m=3,\\
s_2^{-(j+\ell)}\gamma_3(G),&\text{ if }m\ge 4,
\end{cases}\\
\mathrm{V}_3(x^jy^\ell\gamma_2(G))&=&s_{m-1}^{wj+z\ell}s_2^{-\ell}\gamma_3(G)=
\begin{cases}
s_2^{wj+(z-1)\ell}\cdot 1,&\text{ if }m=3,\\
s_2^{-\ell}\gamma_3(G),&\text{ if }m\ge 4.
\end{cases}
\end{eqnarray*}

\noindent
To determine the transfer kernel
\(\mathrm{Ker}(\mathrm{V}_i)\) for \(1\le i\le 3\),
we have to solve the equation
\(\mathrm{V}_i(x^jy^\ell\gamma_2(G))=1\cdot\gamma_2(M_i)\)
with respect to the element \(x^jy^\ell\), that is, with respect to \(j,\ell\),
for the corresponding image of the transfer \(\mathrm{V}_i\).

Since the image of the first transfer \(\mathrm{V}_1\) is regular,
we can use the result for \(\mathrm{Ker}(\mathrm{V}_1)\)
of theorem \ref{t:MaxKerTrf} with \(k=0\).

The images of the second and third transfer
\(\mathrm{V}_2\), \(\mathrm{V}_3\)
are independent from the relational exponents \(w,z\) in the case \(m\ge 4\),
that is, they are equal for all groups with index of nilpotency \(m\ge 4\).\\
The equation
\(\mathrm{V}_2(x^jy^\ell\gamma_2(G))=s_2^{-(j+\ell)}\gamma_3(G)=1\cdot \gamma_3(G)\)
enforces \(j+\ell=0\), that is \(j=\ell\),
because we are dealing with \(2\)-groups,
and \(\langle xy,\gamma_2(G)\rangle=M_3\) implies
\(\mathrm{Ker}(\mathrm{V}_2)=M_3/\gamma_2(G)\), \(\varkappa(2)=3\).\\
The equation
\(\mathrm{V}_3(x^jy^\ell\gamma_2(G))=s_2^{-\ell}\gamma_3(G)=1\cdot \gamma_3(G)\)
implies \(\ell=0\), whereas \(j\) remains arbitrary.
Since \(\langle x,\gamma_2(G)\rangle=M_2\), we obtain
\(\mathrm{Ker}(\mathrm{V}_3)=M_2/\gamma_2(G)\), \(\varkappa(3)=2\).

For \(m=3\), we have only two isomorphism classes
of metabelian \(2\)-groups of maximal class,
\(G_0^{(3)}(0,0)=D(8)\) and
\(G_0^{(3)}(0,1)=Q(8)\).

In the case \(w=0\), \(z=0\), we obtain\\
\(\mathrm{V}_2(x^jy^\ell\gamma_2(G))=s_2^{(w-1)j+(z-1)\ell}\cdot 1=s_2^{-(j+\ell)}\)
and\\
\(\mathrm{V}_3(x^jy^\ell\gamma_2(G))=s_2^{wj+(z-1)\ell}\cdot 1=s_2^{-\ell}\)
and thus, similarly as for \(m\ge 4\),\\
\(\mathrm{Ker}(\mathrm{V}_2)=M_3/\gamma_2(G)\), \(\varkappa(2)=3\)
and \(\mathrm{Ker}(\mathrm{V}_3)=M_2/\gamma_2(G)\), \(\varkappa(3)=2\).

In the case \(w=1\), \(z=0\), however, we have\\
\(\mathrm{V}_2(x^jy^\ell\gamma_2(G))=s_2^{(w-1)j+(z-1)\ell}\cdot 1=s_2^{-\ell}\)
and\\
\(\mathrm{V}_3(x^jy^\ell\gamma_2(G))=s_2^{wj+(z-1)\ell}\cdot 1=s_2^{j-\ell}\)
and thus, conversely,\\
\(\mathrm{Ker}(\mathrm{V}_2)=M_2/\gamma_2(G)\), \(\varkappa(2)=2\)
and \(\mathrm{Ker}(\mathrm{V}_3)=M_3/\gamma_2(G)\), \(\varkappa(3)=3\).\\
\end{proof}

In table \ref{tab:MltTrf} we compare
the multiplets of transfer types
of metabelian \(p\)-groups of maximal class
for \(p=2\) and \(p\ge 3\).
Here, \(\kappa\) denotes the multiplet of coarse transfer types
expressed with the aid of condition (B) by Taussky \cite[p.435]{Ta},
which is given by Kisilevsky \cite[p.273, Th.2]{Ki2}
and by Benjamin and Snyder \cite[p.163, \S 2]{BeSn}.
Further, \(1\le z\le p-1\) is one of the exponents in the relations (\ref{eqn:MaxRel}).
Examples of number fields with these principalisation types are given in \cite{Ma, Ma2}.

\renewcommand{\arraystretch}{1.5}
\begin{table}[ht]
\caption{Transfer types of corresponding \(p\)-groups for \(p=2\) and \(p\ge 3\)}
\label{tab:MltTrf}
\begin{center}
\begin{tabular}{|ccc|c|cc|}
\hline
 \multicolumn{3}{|c|}{\(p=2\)}                          &              & \multicolumn{2}{|c|}{\(p\ge 3\)}                                      \\
\hline
 \(\kappa\) & \(\varkappa\) & \(2\)-group               & \(m\)        & \(\varkappa\)                                 & \(p\)-group           \\
\hline
 \((000)\)  & \((000)\)     & \(C(2)\times C(2)\)       & \(2\)        & \((\overbrace{0,\ldots,0}^{p+1\text{ times}})\) & \(C(p)\times C(p)\) \\
 \((123)\)  & \((123)\)     & \(G_0^{(3)}(0,1)=Q(8)\)   & \(3\)        & \((\overbrace{1,\ldots,1}^{p+1\text{ times}})\) & \(G_0^{(3)}(0,1)\)  \\
 \((0BB)\)  & \((032)\)     & \(G_0^{(m)}(0,0)=D(2^m)\) & \(\ge 3\)    & \((\overbrace{0,\ldots,0}^{p+1\text{ times}})\) & \(G_0^{(m)}(0,0)\)  \\
 \((1BB)\)  & \((132)\)     & \(G_0^{(m)}(0,1)=Q(2^m)\) & \(\ge 4\)    & \((1,\overbrace{0,\ldots,0}^{p\text{ times}})\) & \(G_0^{(m)}(0,1)\)  \\
 \((BBB)\)  & \((232)\)     & \(G_0^{(m)}(1,0)=S(2^m)\) & \(\ge 4\)    & \((2,\overbrace{0,\ldots,0}^{p\text{ times}})\) & \(G_0^{(m)}(z,0)\)  \\
\hline
\end{tabular}
\end{center}
\end{table}


\subsection{Combinatorially possible transfer types of \(2\)-groups}
\label{ss:DyadicTrf}

In this subsection, we arrange the combinatorially possible
\(S_3\)-orbits of the \(4^3\) triplets \(\varkappa\in\lbrack 0,3\rbrack^3\)
according to increasing cardinality of the image and decreasing number of fixed points.
Table \ref{tab:TwoPrtTrf} shows the partial triplets
and table \ref{tab:TwoTotTrf} the total triplets as the possible transfer types
of metabelian \(2\)-groups \(G\) with \(G/\gamma_2(G)\) of type \((2,2)\),
resp. principalisation types
of base fields \(K\) with \(2\)-class group \(\mathrm{Cl}_2(K)\) of type \((2,2)\).
The orbits are divided into sections, denoted by letters,
and identified by ordinal numbers.

We denote by \(o(\varkappa)=(\vert\varkappa^{-1}\lbrace i\rbrace\vert)_{0\le i\le 3}\)
the family of occupation numbers of the selected orbit representative \(\varkappa\) and by
\(F=\lbrace 1\le i\le 3\mid\varkappa(i)=i\rbrace\) the set of fixed points of \(\varkappa\).
In the characterising property, needed for equal numbers of fixed points, let
\(I=\lbrace\varkappa(i)\mid 1\le i\le 3\rbrace\) be the image of \(\varkappa\) and
\(Z=\varkappa^{-1}\lbrace 0\rbrace\) the preimage of zero under \(\varkappa\).

If an orbit can be realised as a transfer type,
then a suitable \(2\)-group \(G\) is given,
according to theorem \ref{t:TwoKerTrf}.

\newpage

\renewcommand{\arraystretch}{1.2}
\begin{table}[ht]
\caption{The seven \(S_3\)-orbits of triplets \(\varkappa\in\lbrack 1,3\rbrack^3\) with \(\nu=0\)}
\label{tab:TwoPrtTrf}
\begin{center}
\begin{tabular}{|rr|c|cccc|c|}
\hline
      &     & repres.       & occupation       & fixed   & charact.                       & cardinality           & realising                      \\
 Sec. & Nr. & of orbit      & numbers          & points  & property                       & of orbit              & \(2\)-group                    \\
      &     & \(\varkappa\) & \(o(\varkappa)\) & \(|F|\) &                                & \(|\varkappa^{S_3}|\) & \(G\)                          \\
\hline
    A &   1 & \((111)\)     & \((0300)\)       & \(1\)   & constant                       &                 \(3\) & impossible                     \\
\hline
    B &   2 & \((121)\)     & \((0210)\)       & \(2\)   & almost                         &                 \(6\) & impossible                     \\
    B &   3 & \((112)\)     & \((0210)\)       & \(1\)   & con-                           &                 \(6\) & impossible                     \\
    S &   4 & \((211)\)     & \((0210)\)       & \(0\)   & stant                          &                 \(6\) & \(G_0^{(m)}(1,0)=S(2^m)\), \(m\ge 4\) \\
\hline
    Q &   5 & \((123)\)     & \((0111)\)       & \(3\)   & identity                       &                 \(1\) & \(G_0^{(3)}(0,1)=Q(8)\), \(m=3\) \\
    Q &   6 & \((132)\)     & \((0111)\)       & \(1\)   & transposition                  &                 \(3\) & \(G_0^{(m)}(0,1)=Q(2^m)\), \(m\ge 4\) \\
    C &   7 & \((231)\)     & \((0111)\)       & \(0\)   & \(3\)-cycle                    &                 \(2\) & impossible                     \\
\hline
      &     &               &                  &         &                  Total number: &                \(27\) &                                \\
\hline
\end{tabular}
\end{center}
\end{table}

\renewcommand{\arraystretch}{1.2}
\begin{table}[ht]
\caption{The nine \(S_3\)-orbits of triplets \(\varkappa\in\lbrack 0,3\rbrack^3\setminus\lbrack 1,3\rbrack^3\) with \(1\le\nu\le 3\)}
\label{tab:TwoTotTrf}
\begin{center}
\begin{tabular}{|rr|c|cccc|c|}
\hline
      &     & repres.       & occupation       & fixed   & charact.                       & cardinality           & realising                      \\
 Sec. & Nr. & of orbit      & numbers          & points  & property                       & of orbit              & \(2\)-group                    \\
      &     & \(\varkappa\) & \(o(\varkappa)\) & \(|F|\) &                                & \(|\varkappa^{S_3}|\) & \(G\)                          \\
\hline
    a &   1 & \((000)\)     & \((3000)\)       & \(0\)   & constant                       &                 \(1\) & \(C(2)\times C(2)\),     \(m=2\) \\
\hline
    b &   2 & \((100)\)     & \((2100)\)       & \(1\)   &                                &                 \(3\) & impossible                     \\
    b &   3 & \((010)\)     & \((2100)\)       & \(0\)   &                                &                 \(6\) & impossible                     \\
\hline
    c &   4 & \((110)\)     & \((1200)\)       & \(1\)   &                                &                 \(6\) & impossible                    \\
    c &   5 & \((011)\)     & \((1200)\)       & \(0\)   &                                &                 \(3\) & impossible                    \\
\hline
    e &   6 & \((120)\)     & \((1110)\)       & \(2\)   & identity with \(0\)            &                 \(3\) & impossible                    \\
    e &   7 & \((021)\)     & \((1110)\)       & \(1\)   &                                &                 \(6\) & impossible                    \\
    d &   8 & \((210)\)     & \((1110)\)       & \(0\)   & \(Z\not\subset I\)             &                 \(3\) & \(G_0^{(m)}(0,0)=D(2^m)\), \(m\ge 3\) \\
    e &   9 & \((012)\)     & \((1110)\)       & \(0\)   & \(Z\subset I\)                 &                 \(6\) & impossible                    \\
\hline
      &     &               &                  &         &                  Total number: &          \(37=64-27\) &                                \\
\hline
\end{tabular}
\end{center}
\end{table}

\noindent
S.4 is the transfer type of the semi-dihedral group \(S(2^m)=G^{(m)}_0(1,0)\) with \(m\ge 4\),
Q.5 the transfer type of the quaternion group \(Q(8)=G^{(3)}_0(0,1)\), \(m=3\),
Q.6 the transfer type of the generalised quaternion group \(Q(2^m)=G^{(m)}_0(0,1)\), \(m\ge 4\),
a.1 the transfer type of the elementary abelian bicyclic \(2\)-group of type \((2,2)\), \(m=2\),
and, finally,
d.8 is the transfer type of the dihedral group \(D(2^m)=G^{(m)}_0(0,0)\) with \(m\ge 3\).


\section{Transfers of a metabelian \(3\)-group of non-maximal class}
\label{s:LowTrf}

The transfer kernels \(\mathrm{Ker}(\mathrm{V}_i)\) and transfer types \(\varkappa(G)\)
of metabelian \(3\)-groups \(G\) of non-maximal class
with abelianisation \(G/\gamma_2(G)\) of type \((3,3)\)
have been determined in Nebelung's thesis \cite{Ne}.
These transfer types are summarised
in the tables \ref{tab:TriPrtTrf} and \ref{tab:TriTotTrf}
of subsection \ref{ss:TriadicTrf}.

However, our concrete numerical investigation \cite{Ma}
of all \(4\,596\) quadratic base fields \(K=\mathbb{Q}(\sqrt{D})\)
with \(3\)-class group of type \((3,3)\) and
discriminant in the range \(-10^6<D<10^7\)
has revealed that supplementary criteria are necessary
to distinguish between two kinds of the transfer types
\(\mathrm{d}.19\), \(\mathrm{d}.23\), and \(\mathrm{d}.25\)
of the second \(3\)-class group \(G=\mathrm{Gal}(\mathrm{F}_3^2(K)\vert K)\)
of arbitrary base fields \(K\).
To be able to describe these two kinds of transfer types of section \(\mathrm{d}\),
we must recall some results of \cite{Ne}.

The metabelian \(3\)-groups \(G\) of non-maximal class,
that is, of coclass \(\mathrm{cc}(G)\ge 2\),
cannot be CF-groups,
since they must have at least one bicyclic factor \(\gamma_3(G)/\gamma_4(G)\).
Similarly as in section \(2\) of \cite{Ma}, we declare
an isomorphism invariant \(e=e(G)\) of \(G\) by
\(e+1=\min\lbrace 3\le j\le m\mid 1\le\lvert\gamma_j(G)/\gamma_{j+1}(G)\rvert\le 3\rbrace\).
This invariant \(2\le e\le m-1\) characterises
the first cyclic factor \(\gamma_{e+1}(G)/\gamma_{e+2}(G)\) of the lower central series of \(G\),
except \(\gamma_2(G)/\gamma_3(G)\), which is always cyclic.
\(e\) can be calculated
from the \(3\)-exponent \(n\) of the order \(\lvert G\rvert=3^n\)
and the class \(\mathrm{cl}(G)=m-1\), resp. the index \(m\), of nilpotency of \(G\)
by the formula \(e=n-m+2\).
Since the coclass of \(G\) is given by \(\mathrm{cc}(G)=n-\mathrm{cl}(G)=n-m+1\),
we have the relation \(e=\mathrm{cc}(G)+1\).

The isomorphism classes of all metabelian \(3\)-groups \(G\)
with abelianisation \(G/\gamma_2(G)\) of type \((3,3)\)
can be represented as nodes of a directed tree
with root \(\mathrm{C}(3)\times\mathrm{C}(3)\) \cite[p.181 ff]{Ne}.
The tree contains infinite chains,
all of whose nodes have the same transfer type.
There exists a single chain of transfer type \(\mathrm{a}.1\),
giving rise to all groups of maximal class \cite{Bl} with invariant \(e=2\).
This is the main line of the unique coclass tree in the
coclass graph \(\mathcal{G}(3,1)\) \cite{As,LgMk,EiLg,DEF}.
Further there are three chains of the transfer types
\(\mathrm{b}.10\), \(\mathrm{c}.21\), and \(\mathrm{c}.18\),
giving rise to all groups of second maximal class \cite{As} with invariant \(e=3\).
These are main lines of coclass trees in coclass graph \(\mathcal{G}(3,2)\).
Eventually, for each \(r\ge 3\),
there are finitely many chains of the transfer types
\(\mathrm{b}.10\), \(\mathrm{d}.23\), \(\mathrm{d}.25\), and \(\mathrm{d}.19\),
giving rise to groups of lower than second maximal class with invariant \(e=r+1\).
These are main lines of coclass trees in coclass graph \(\mathcal{G}(3,r)\)
with a certain periodicity of length two with respect to \(r\).

Only the transfer types \(\mathrm{d}.19\), \(\mathrm{d}.23\), and \(\mathrm{d}.25\)
can occur for both,
\textit{internal} nodes on the corresponding chains with \(e\ge 4\)
and \textit{terminal} successors
of nodes on all chains with transfer type \(\mathrm{b}.10\).
We will show that
the distinction of internal and terminal nodes is generally possible
with the aid of the canonical multiplet \(\varkappa\) of transfer types of \(G\)
for arbitrary base fields \(K\),
and in particular by means of the parity of the index \(m\) of nilpotency of \(G\)
for quadratic base fields \(K=\mathbb{Q}(\sqrt{D})\).


\subsection{Images of the transfers}
\label{ss:LowTrmTrf}

For a group \(G\) of non-maximal class
we need a generalisation of the group \(\chi_2(G)\).
Denoting by \(m\) the index of nilpotency of \(G\),
we let \(\chi_j(G)\) with \(2\le j\le m-1\)
be the centralisers
of two-step factor groups \(\gamma_j(G)/\gamma_{j+2}(G)\)
of the lower central series, that is,
the biggest subgroups of \(G\) with the property
\(\lbrack\chi_j(G),\gamma_j(G)\rbrack\le\gamma_{j+2}(G)\).
They form an ascending chain of characteristic subgroups of \(G\),
\(\gamma_2(G)\le\chi_2(G)\le\ldots\le\chi_{m-2}(G)<\chi_{m-1}(G)=G\),
which contain the commutator group \(\gamma_2(G)\).
\(\chi_j(G)\) coincides with \(G\), if and only if \(j\ge m-1\).
Similarly as in section 2 of \cite{Ma},
we characterise the smallest two-step centraliser
different from the commutator group
by an isomorphism invariant
\(s=s(G)=\min\lbrace 2\le j\le m-1\mid\chi_j(G)>\gamma_2(G)\rbrace\).

The assumptions of the following theorem \ref{t:LowTrmTrf}
for a metabelian \(3\)-group \(G\) of non-maximal class
with abelianisation \(G/\gamma_2(G)\) of type \((3,3)\)
can always be satisfied, according to \cite{Ne}.

\begin{theorem}
\label{t:LowTrmTrf}

Let \(G\) be a metabelian \(3\)-group of non-maximal class
with abelianisation \(G/\gamma_2(G)\) of type \((3,3)\).
Assume that \(G\) has order \(\lvert G\rvert=3^n\),
class \(\mathrm{cl}(G)=m-1\), and invariant \(e=n-m+2\ge 3\),
where \(4\le m<n\le 2m-3\).
Let generators of \(G=\langle x,y\rangle\) be selected such that
\(\gamma_3(G)=\langle y^3,x^3,\gamma_4(G)\rangle\),
\(x\in G\setminus\chi_s(G)\), if \(s<m-1\),
and \(y\in\chi_s(G)\setminus\gamma_2(G)\).
Suppose that the order of the four maximal normal subgroups of \(G\) is defined by
\(M_i=\langle g_i,\gamma_2(G)\rangle\) with
\(g_1=y\), \(g_2=x\), \(g_3=xy\), and \(g_4=xy^{-1}\).
Let the main commutator of \(G\) be declared by
\(s_2=t_2=\lbrack y,x\rbrack\in\gamma_2(G)\)
and higher commutators recursively by
\(s_j=\lbrack s_{j-1},x\rbrack\), \(t_j=\lbrack t_{j-1},y\rbrack\in\gamma_j(G)\)
for \(j\ge 3\).
Starting with the powers \(\sigma_3=y^3\), \(\tau_3=x^3\in\gamma_3(G)\), let
\(\sigma_j=\lbrack\sigma_{j-1},x\rbrack\), \(\tau_j=\lbrack\tau_{j-1},y\rbrack\in\gamma_j(G)\)
for \(j\ge 4\).
With exponents \(-1\le\alpha,\beta,\gamma,\delta,\rho\le 1\), let
the following relations be satisfied

\begin{equation}
\label{eqn:LowRel}
s_2^3=\sigma_4\sigma_{m-1}^{-\rho\beta}\tau_4^{-1},\quad
\ s_3\sigma_3\sigma_4=\sigma_{m-2}^{\rho\beta}\sigma_{m-1}^\gamma\tau_e^\delta,\quad
\ t_3^{-1}\tau_3\tau_4=\sigma_{m-2}^{\rho\delta}\sigma_{m-1}^\alpha\tau_e^\beta,\quad
\ \tau_{e+1}=\sigma_{m-1}^{-\rho}.
\end{equation}

\noindent
Finally, let \(\lbrack\chi_s(G),\gamma_e(G)\rbrack=\gamma_{m-k}(G)\)
with \(0\le k\le 1\).\\
If the cosets are represented in the form
\(g\equiv x^jy^\ell\pmod{\gamma_2(G)}\) with \(-1\le j,\ell\le 1\),
then the images of the transfers
\(\mathrm{V}_i:G/\gamma_2(G)\longrightarrow M_i/\gamma_2(M_i)\)
are given by

\begin{eqnarray*}
\label{e:LowTrfDtl}
\mathrm{V}_1(g\gamma_2(G))&=&
\begin{cases}
\sigma_{m-1}^{\gamma\ell}\tau_e^{\delta\ell}\tau_3^j\gamma_2(M_1),
&\text{ if }\lbrack\chi_s(G),\gamma_e(G)\rbrack=1,\ k=0,\\
\sigma_{m-2}^{\rho\beta\ell}\sigma_{m-1}^{(\gamma-\rho\beta)\ell}\tau_e^{\delta\ell}\tau_3^j\gamma_2(M_1),
&\text{ if }\lbrack\chi_s(G),\gamma_e(G)\rbrack=\gamma_{m-1}(G),\ k=1,\ m\ge 5,
\end{cases}\\
\mathrm{V}_2(g\gamma_2(G))&=&
\begin{cases}
\sigma_3^\ell\sigma_{m-1}^{\alpha j}\tau_e^{\beta j}\gamma_2(M_2),
&\text{ if }\lbrack\chi_s(G),\gamma_e(G)\rbrack=1,\ k=0,\\
\sigma_3^\ell\sigma_{m-2}^{\rho\delta j}\sigma_{m-1}^{(\alpha+\rho\beta)j}\tau_e^{\beta j}\gamma_2(M_2),
&\text{ if }\lbrack\chi_s(G),\gamma_e(G)\rbrack=\gamma_{m-1}(G),\ k=1,\ m\ge 5,
\end{cases}\\
\mathrm{V}_i(g\gamma_2(G))&=&
\sigma_3^\ell\tau_3^j\gamma_2(M_i)
\text{ for }3\le i\le 4.
\end{eqnarray*}

\end{theorem}

\begin{proof}

First we prove a formula for the third trace element
acting as symbolic exponent.
Let \(G\) be an arbitrary group
with a normal subgroup \(N<G\).
For \(u\in N\) and \(h\in G\setminus N\) we have

\begin{equation}
\label{eqn:LowTrf}
u^{\mathrm{S}_3(h)}
\equiv u^3\lbrack u,h\rbrack^3\lbrack\lbrack u,h\rbrack,h\rbrack
\pmod{\gamma_2(N)}.
\end{equation}

\noindent
This follows from
\(u^{\mathrm{S}_3(h)}=u^{1+h+h^2}=u\cdot h^{-1}uh\cdot h^{-2}uh^2\) and\\
\(u^3\lbrack u,h\rbrack^3\lbrack\lbrack u,h\rbrack,h\rbrack
=u^3\lbrack u,h\rbrack^3\lbrack u,h\rbrack^{-1}h^{-1}\lbrack u,h\rbrack h\)\\
\(=u^3\lbrack u,h\rbrack^2h^{-1}\cdot u^{-1}h^{-1}uh\cdot h
=u^3\cdot u^{-1}h^{-1}uh\cdot u^{-1}h^{-1}uh\cdot h^{-1}\cdot u^{-1}h^{-1}uh^2\)\\
\(=u^2\cdot h^{-1}uhu^{-1}\cdot h^{-2}uh^2
=u^2\cdot u^{-1}h^{-1}uh\lbrack h^{-1}uh,u^{-1}\rbrack\cdot h^{-2}uh^2\)\\
\(\equiv u\cdot h^{-1}uh\cdot h^{-2}uh^2\pmod{\gamma_2(N)}\),
since \(h^{-1}uh,u^{-1}\in N\), and thus \(\lbrack h^{-1}uh,u^{-1}\rbrack\in\lbrack N,N\rbrack=\gamma_2(N)\).

For the first transfer \(\mathrm{V}_1\) we have
\(x\in G\setminus M_1\) and \(y\in M_1\), and thus\\
\(\mathrm{V}_1(x\gamma_2(G))=x^3\gamma_2(M_1)=\tau_3\gamma_2(M_1)\) for the generator \(x\).\\
Now we use formula(\ref{eqn:LowTrf}) and the relations (\ref{eqn:LowRel}) for the generator \(y\),\\
\(\mathrm{V}_1(y\gamma_2(G))=y^{\mathrm{S}_3(x)}\gamma_2(M_1)
=y^3\lbrack y,x\rbrack^3\lbrack\lbrack y,x\rbrack,x\rbrack\gamma_2(M_1)
=y^3 s_2^3s_3\gamma_2(M_1)\)\\
\(=\sigma_3\cdot \sigma_4\sigma_{m-1}^{-\rho\beta}\tau_4^{-1}\cdot \sigma_3^{-1}\sigma_4^{-1}\sigma_{m-2}^{\rho\beta}\sigma_{m-1}^\gamma\tau_e^\delta\gamma_2(M_1)
=\sigma_{m-2}^{\rho\beta}\sigma_{m-1}^{\gamma-\rho\beta}\tau_e^\delta\gamma_2(M_1)\),\\
since \(\tau_4\in\gamma_2(M_1)=\langle t_3,\tau_4,\ldots,\tau_{e+1}\rangle\),
by \cite[Cor.4.1.1]{Ma}.

For the second transfer \(\mathrm{V}_2\) we have
\(x\in M_2\) und \(y\in G\setminus M_2\), and thus\\
\(\mathrm{V}_2(y\gamma_2(G))=y^3\gamma_2(M_2)=\sigma_3\gamma_2(M_2)\) for the generator \(y\).\\
Now we use formula (\ref{eqn:LowTrf}) and the relations (\ref{eqn:LowRel}) for the generator \(x\),\\
\(\mathrm{V}_2(x\gamma_2(G))=x^{\mathrm{S}_3(y)}\gamma_2(M_2)
=x^3\lbrack x,y\rbrack^3\lbrack\lbrack x,y\rbrack,y\rbrack\gamma_2(M_2)
=x^3s_2^{-3}t_3^{-1}\gamma_2(M_2)\)\\
\(=\tau_3\cdot \sigma_4^{-1}\sigma_{m-1}^{\rho\beta}\tau_4\cdot \tau_3^{-1}\tau_4^{-1}\sigma_{m-2}^{\rho\delta}\sigma_{m-1}^\alpha\tau_e^\beta\gamma_2(M_2)
=\sigma_{m-2}^{\rho\delta}\sigma_{m-1}^{\alpha+\rho\beta}\tau_e^\beta\gamma_2(M_2)\),\\
since
\(\lbrack\lbrack x,y\rbrack,y\rbrack=\lbrack s_2^{-1},y\rbrack
=\lbrack s_2,y\rbrack^{-s_2^{-1}}=\left(t_3^{-1}\right)^{s_2^{-1}}=t_3^{-1}\), by the formula for inverses,\\
and \(\sigma_4\in\gamma_2(M_2)=\langle s_3,\sigma_4,\ldots,\sigma_{m-1}\rangle\),
by \cite[Cor.4.1.1]{Ma}.

Thus, the images for an arbitrary coset
\(g\equiv x^jy^\ell\pmod{\gamma_2(G)}\) with \(-1\le j,\ell\le 1\)
are\\
\(\mathrm{V}_1(g\gamma_2(G))=
\tau_3^j\sigma_{m-2}^{\rho\beta\ell}\sigma_{m-1}^{(\gamma-\rho\beta)\ell}\tau_e^{\delta\ell}\gamma_2(M_1)\),\\
\(\mathrm{V}_2(g\gamma_2(G))=
\sigma_{m-2}^{\rho\delta j}\sigma_{m-1}^{(\alpha+\rho\beta)j}\tau_e^{\beta j}\sigma_3^\ell\gamma_2(M_2)\).

The simplified images for \(k=0\) are a consequence of the central relation
\(\lbrack y,\tau_e\rbrack=\tau_{e+1}^{-1}=\sigma_{m-1}^\rho\) in
\(\lbrack\chi_s(G),\gamma_e(G)\rbrack=\gamma_{m-k}(G)\le\zeta_1(G)\),
since \(k=0\) is equivalent with \(\rho=0\).

For the other two transfers \(\mathrm{V}_i\) with \(3\le i\le 4\), we have
\(x\in G\setminus M_i\) and \(y\in G\setminus M_i\), and thus\\
\(\mathrm{V}_i(x\gamma_2(G))=x^3\gamma_2(M_i)=\tau_3\gamma_2(M_i)\),\\
\(\mathrm{V}_i(y\gamma_2(G))=y^3\gamma_2(M_i)=\sigma_3\gamma_2(M_i)\), and therefore
\(\mathrm{V}_i(g\gamma_2(G))=\tau_3^j\sigma_3^\ell\gamma_2(M_i)\).\\
According to \cite{Ne},
the possible maximum of invariant \(k\) depends on the index of nilpotency \(m\).
\end{proof}


\subsection{Kernels of the transfers}
\label{ss:LowKerTrf}

Since we are particularly interested in
second \(3\)-class groups \(G\) of quadratic number fields \(K=\mathbb{Q}(\sqrt{D})\)
with transfer types in section \(\mathrm{d}\),
we now focus on metabelian \(3\)-groups \(G\)
of coclass \(\mathrm{cc}(G)\ge 3\).

\begin{theorem}
\label{t:LowKerTrf}

Let \(G\) be a metabelian \(3\)-group,
having abelianisation \(G/\gamma_2(G)\) of type \((3,3)\),
of lower than second maximal class, \(e\ge 4\),
and thus with index of nilpotency \(m\ge 6\) or \(m=5\), \(k=0\).
For the generators of \(G=\langle x,y\rangle\)
and the order of the maximal normal subgroups \(M_1,\ldots,M_4\)
let the assumptions of theorem \ref{t:LowTrmTrf} be satisfied.
Suppose that \(G\) is defined by
the relations (\ref{eqn:LowRel}) 
with a system of exponents \(-1\le\alpha,\beta,\gamma,\delta,\rho\le 1\)
and that
\(\lbrack\chi_s(G),\gamma_e(G)\rbrack=\gamma_{m-k}(G)\)
with \(0\le k\le 1\).\\
Finally, let a function \(f\) of two parameters \(-1\le\lambda,\mu\le 1\)
be defined by
\[f(\lambda,\mu)=
\begin{cases}
0,&\text{ if }\lambda=0,\ \mu=0,\\
1,&\text{ if }\lambda=\pm 1,\ \mu=0,\\
2,&\text{ if }\lambda=0,\ \mu=\pm 1,\\
3,&\text{ if }\lambda=\pm 1,\ \mu=\pm 1,\ \lambda=-\mu,\\
4,&\text{ if }\lambda=\pm 1,\ \mu=\pm 1,\ \lambda=\mu.\\
\end{cases}\]
Then the singulets of transfer types of \(G\) are given by

\begin{eqnarray*}
\varkappa(1)&=&
\begin{cases}
f(\alpha,\gamma),&\text{ if }k=0,\\
f(\delta,\beta),&\text{ if }k=1,\\
\end{cases}\\
\varkappa(2)&=&f(\beta,\delta),\\
\varkappa(3)&=&4,\\
\varkappa(4)&=&3.
\end{eqnarray*}

\noindent
Therefore, the kernels of the transfers
\(\mathrm{V}_i:G/\gamma_2(G)\longrightarrow M_i/\gamma_2(M_i)\)
with \(1\le i\le 4\)
are given by\\
\(\mathrm{Ker}(\mathrm{V}_i)=M_{\varkappa(i)}/\gamma_2(G)\),
if \(1\le\varkappa(i)\le 4\), and
\(\mathrm{Ker}(\mathrm{V}_i)=G/\gamma_2(G)\),
if \(\varkappa(i)=0\).

\end{theorem}

\begin{proof}

We start with Corollary 4.1.1 of our paper \cite{Ma},
where the commutator groups \(\gamma_2(M_i)\)
of the maximal normal subgroups \(M_i\) of \(G\) are determined:

\begin{eqnarray*}
\gamma_2(M_1)&=&\langle t_3,\tau_4,\ldots,\tau_{e+1}\rangle,\\
\gamma_2(M_2)&=&\langle s_3,\sigma_4,\ldots,\sigma_{m-1}\rangle,\\
\gamma_2(M_3)&=&\langle s_3t_3,\gamma_4(G)\rangle,\\
\gamma_2(M_4)&=&\langle s_3t_3^{-1},\gamma_4(G)\rangle,
\end{eqnarray*}

\noindent
where \(\gamma_4(G)=\langle\sigma_4,\ldots,\sigma_{m-1},\tau_4,\ldots,\tau_{e+1}\rangle\).

As a consequence of the relations (\ref{eqn:LowRel}), we have

\begin{eqnarray*}
s_3&=&\sigma_3^{-1}\sigma_4^{-1}\sigma_{m-2}^{\rho\beta}\sigma_{m-1}^\gamma\tau_e^\delta,\\
t_3&=&\tau_3\tau_4\sigma_{m-2}^{-\rho\delta}\sigma_{m-1}^{-\alpha}\tau_e^{-\beta},\\
s_3t_3&=&\sigma_3^{-1}\sigma_4^{-1}\sigma_{m-2}^{\rho(\beta-\delta)}\sigma_{m-1}^{\gamma-\alpha}\tau_3\tau_4\tau_e^{\delta-\beta},\\
s_3t_3^{-1}&=&\sigma_3^{-1}\sigma_4^{-1}\sigma_{m-2}^{\rho(\beta+\delta)}\sigma_{m-1}^{\alpha+\gamma}\tau_3^{-1}\tau_4^{-1}\tau_e^{\beta+\delta}.
\end{eqnarray*}

To determine the kernel \(\mathrm{Ker}(\mathrm{V}_i)\)
of the transfer \(\mathrm{V}_i:G/\gamma_2(G)\longrightarrow M_i/\gamma_2(M_i)\) for \(1\le i\le 4\),
where cosets \(g\gamma_2(G)\in G/\gamma_2(G)\) are represented by the generators \(x,y\) in the shape
\(g\equiv x^jy^\ell\pmod{\gamma_2(G)}\) with \(-1\le j,\ell\le 1\),
the equation 
\(\mathrm{V}_i(g\gamma_2(G))=\mathrm{V}_i(x^jy^\ell\gamma_2(G))=1\cdot\gamma_2(M_i)\)
must be solved with respect to \(j,\ell\).
For this purpose, we use the images of the transfers in theorem \ref{t:LowTrmTrf}.

For the kernel of the first transfer we have\\
\(\mathrm{V}_1(x^jy^\ell\gamma_2(G))
=\sigma_{m-2}^{\rho\beta\ell}\sigma_{m-1}^{(\gamma-\rho\beta)\ell}\tau_e^{\delta\ell}\tau_3^j\gamma_2(M_1)=1\cdot\gamma_2(M_1)\),\\
and thus
\(\left(\sigma_{m-2}^{\rho\beta}\sigma_{m-1}^{\gamma-\rho\beta}\tau_e^\delta\right)^\ell\tau_3^j\in\gamma_2(M_1)
=\langle t_3,\tau_4,\ldots,\tau_{e+1}\rangle
=\langle\tau_3\tau_4\sigma_{m-2}^{-\rho\delta}\sigma_{m-1}^{-\alpha}\tau_e^{-\beta},\tau_4,\ldots,\tau_{e+1}\rangle\).\\
For \(e\ge 4\) we obtain a simplification\\
\(\left(\sigma_{m-2}^{\rho\beta}\sigma_{m-1}^{\gamma-\rho\beta}\right)^\ell\tau_3^j\in
\langle\tau_3\sigma_{m-2}^{-\rho\delta}\sigma_{m-1}^{-\alpha},\tau_4,\ldots,\tau_{e+1}\rangle\),\\
and thus
\(\left(\sigma_{m-2}^{\rho\beta}\sigma_{m-1}^{\gamma-\rho\beta}\right)^\ell\tau_3^j
=\left(\sigma_{m-2}^{-\rho\delta}\sigma_{m-1}^{-\alpha}\tau_3\right)^r\)
with \(-1\le r\le 1\).\\
Now we distiguish the values of the invariant \(0\le k\le 1\).\\
In the case \(k=0\), that is, for \(\rho=0\), we have
\(\sigma_{m-1}^{\gamma\ell}\tau_3^j
=\left(\sigma_{m-1}^{-\alpha}\tau_3\right)^r\),
therefore \(j=r\), \(\gamma\ell=-\alpha r\),
and consequently \(\alpha j=-\gamma\ell\).\\
In the case \(k=1\), however, that is, for \(\rho=\pm 1\), we have
\(\sigma_{m-1}=\tau_{e+1}^{-\rho}\) with \(e+1\ge 5\) and thus\\
\(\sigma_{m-2}^{\rho\beta\ell}\tau_3^j
=\left(\sigma_{m-2}^{-\rho\delta}\tau_3\right)^r\),
therefore \(j=r\), \(\rho\beta\ell=-\rho\delta r\),
and consequently \(\delta j=-\beta\ell\).

For the kernel of the second transfer we have\\
\(\mathrm{V}_2(x^jy^\ell\gamma_2(G))
=\sigma_3^\ell\sigma_{m-2}^{\rho\delta j}\sigma_{m-1}^{(\alpha+\rho\beta)j}\tau_e^{\beta j}\gamma_2(M_2)=1\cdot\gamma_2(M_2)\),\\
and thus
\(\sigma_3^\ell\left(\sigma_{m-2}^{\rho\delta}\sigma_{m-1}^{\alpha+\rho\beta}\tau_e^\beta\right)^j\in\gamma_2(M_2)
=\langle s_3,\sigma_4,\ldots,\sigma_{m-1}\rangle
=\langle\sigma_3^{-1}\sigma_4^{-1}\sigma_{m-2}^{\rho\beta}\sigma_{m-1}^\gamma\tau_e^\delta,\sigma_4,\ldots,\sigma_{m-1}\rangle\).\\
For \(m\ge 6\) or \(m=5\), \(\rho=0\) we obtain the simplification\\
\(\sigma_3^\ell\tau_e^{\beta j}\in
\langle\sigma_3^{-1}\tau_e^\delta,\sigma_4,\ldots,\sigma_{m-1}\rangle\),\\
and thus
\(\sigma_3^\ell\tau_e^{\beta j}=\left(\sigma_3^{-1}\tau_e^\delta\right)^r\)
with \(-1\le r\le 1\), independently from \(\alpha,\gamma,\rho\).\\
Consequently \(\ell=-r\), \(\beta j=\delta r\) and thus
\(\beta j=-\delta\ell\).

The solution of the equation \(\lambda j=-\mu\ell\) for \(1\le i\le 2\),
which depends on the relational exponents
\((\lambda,\mu)\in\lbrace (\alpha,\gamma),(\delta,\beta),(\beta,\delta)\rbrace\),
is obtained by the following distinction of cases.\\
In the case \(\lambda=0\), \(\mu=0\), we can select \(j,\ell\) arbitrarily,
and the kernel is \(G/\gamma_2(G)\), \(\varkappa(i)=0\).\\
In the case \(\lambda=\pm 1\), \(\mu=0\)
we must have \(j=0\) but \(\ell\) remains arbitrary.
Thus, the kernel equals \(M_1/\gamma_2(G)\), \(\varkappa(i)=1\).\\
In the case \(\lambda=0\), \(\mu=\pm 1\),
we must have \(\ell=0\) but \(j\) remains arbitrary.
Thus, the kernel equals \(M_2/\gamma_2(G)\), \(\varkappa(i)=2\).\\
In the case \(\lambda=\pm 1\), \(\mu=\pm 1\), finally,
we obtain that \(\ell=j\) for \(\lambda=-\mu\)
and the kernel is \(M_3/\gamma_2(G)\), \(\varkappa(i)=3\).
For \(\lambda=\mu\), however, we must have \(\ell=-j\)
and the kernel equals \(M_4/\gamma_2(G)\), \(\varkappa(i)=4\).

For the kernel of the third transfer, we have
\(\mathrm{V}_3(x^jy^\ell\gamma_2(G))=\sigma_3^\ell\tau_3^j\gamma_2(M_3)=1\cdot\gamma_2(M_3)\) and\\
\(\sigma_3^\ell\tau_3^j\in\gamma_2(M_3)=\langle s_3t_3,\gamma_4(G)\rangle
=\langle\sigma_3^{-1}\sigma_4^{-1}\sigma_{m-2}^{\rho(\beta-\delta)}\sigma_{m-1}^{\gamma-\alpha}\tau_3\tau_4\tau_e^{\delta-\beta},\sigma_4,\ldots,\sigma_{m-1},\tau_4,\ldots,\tau_{e+1}\rangle\).\\
In the case of a group of lower than second maximal class
with \(e\ge 4\), and thus either \(m\ge 6\) or \(m=5\), \(\rho=0\),
this relation becomes simply\\
\(\sigma_3^\ell\tau_3^j\in\langle\sigma_3^{-1}\tau_3,\gamma_4(G)\rangle\),
and thus
\(\sigma_3^\ell\tau_3^j=\left(\sigma_3^{-1}\tau_3\right)^r\) with \(-1\le r\le 1\).\\
This means \(\ell=-r\), \(j=r\) and thus\\
\(\ell=-j\),
\(g\equiv\left(xy^{-1}\right)^j\pmod{\gamma_2(G)}\),
\(g\in\langle xy^{-1},\gamma_2(G)\rangle=M_4\),
that is, \(\varkappa(3)=4\).

For the kernel of the fourth transfer, we have
\(\mathrm{V}_4(x^jy^\ell\gamma_2(G))=\sigma_3^\ell\tau_3^j\gamma_2(M_4)=1\cdot\gamma_2(M_4)\) and\\
\(\sigma_3^\ell\tau_3^j\in\gamma_2(M_4)=\langle s_3t_3^{-1},\gamma_4(G)\rangle
=\langle\sigma_3^{-1}\sigma_4^{-1}\sigma_{m-2}^{\rho(\beta+\delta)}\sigma_{m-1}^{\alpha+\gamma}\tau_3^{-1}\tau_4^{-1}\tau_e^{\beta+\delta},\sigma_4,\ldots,\sigma_{m-1},\tau_4,\ldots,\tau_{e+1}\rangle\).\\
For \(e\ge 4\), and thus either \(m\ge 6\) or \(m=5\), \(\rho=0\),
this relation becomes simply\\
\(\sigma_3^\ell\tau_3^j\in\langle\sigma_3^{-1}\tau_3^{-1},\gamma_4(G)\rangle\),
and thus
\(\sigma_3^\ell\tau_3^j=\left(\sigma_3^{-1}\tau_3^{-1}\right)^r\) with \(-1\le r\le 1\).\\
Consequently \(\ell=-r\), \(j=-r\) and thus\\
\(\ell=j\),
\(g\equiv(xy)^j\pmod{\gamma_2(G)}\),
\(g\in\langle xy,\gamma_2(G)\rangle=M_3\),
that is, \(\varkappa(4)=3\).

The kernels of the third and fourth transfer
are independent from the relational exponents \(\alpha,\beta,\gamma,\delta,\rho\)
and from the invariant \(k\).
\end{proof}

The following corollary \ref{c:LowKerTrf}
has been proved by Nebelung \cite[p.200, Lem.6.7; p.204, Lem.6.10]{Ne}
without the use of explicit images of the transfers.
We show that it can also be derived very easily
from theorem \ref{t:LowKerTrf}
for groups of lower than second maximal class.

\begin{corollary}
\label{c:LowKerTrf}

On the directed rooted tree
of all metabelian \(3\)-groups
with abelianisation of type \((3,3)\),
every terminal group with invariant \(k=1\)
has the same transfer type
as its internal predecessor with invariant \(k=0\).

\end{corollary}

\begin{proof}

Let \(G\) and \(H\) be two metabelian \(3\)-groups
with abelianisation of type \((3,3)\),
both of lower than second maximal class.

Assume that \(G\) has index of nilpotency \(m=m(G)\ge 6\).
If \(G\) has the invariant \(k(G)=1\),
then \(G\) is represented by a terminal node in the tree \cite{Ne}
and its center \(\zeta_1(G)=\gamma_{m-1}(G)\) is cyclic of order \(3\).

Suppose that \(H\simeq G/\gamma_{m-1}(G)\)
is the immediate predecessor of \(G\)
and is therefore represented by an internal node in the tree.
Then \(H\) has invariant \(k(H)=0\),
index of nilpotency \(m(H)=m-1\ge 5\),
and the same invariant \(e=e(H)=e(G)\ge 4\) as \(G\).

If \((\alpha,\beta,\gamma,\delta,\rho)\) with \(\rho=\pm 1\)
are the relational exponents of \(G\) in equation (\ref{eqn:LowRel}),
then those of \(H\) are given by
\((\rho\delta,\beta,\rho\beta,\delta,0)\),
according to \cite[p.184, Lem.5.2.6]{Ne}.
 
By theorem \ref{t:LowKerTrf}, we have the singulets of transfer types
\(\varkappa(3)=4\) and \(\varkappa(4)=3\),
for \(G\) as well as for \(H\).
Since the parameters \((\beta,\delta)\) of \(G\) and \(H\) are the same,
we also have \(\varkappa(2)=f(\beta,\delta)\), for both groups.
Finally, the first singulet is given by
\(\varkappa(1)=f(\delta,\beta)\) for \(G\) with \(k(G)=1\),
and by \(\varkappa(1)=f(\rho\delta,\rho\beta)\) for \(H\) with \(k(H)=0\).
However, since \(\rho=\pm 1\), the values
\(f(\rho\delta,\rho\beta)=f(\delta,\beta)\) coincide.

This corollary is also valid
for metabelian \(3\)-groups of maximal and second maximal class
\cite{Ne}.
\end{proof}

\noindent
With the aid of Nebelung's
isomorphism classes of metabelian \(3\)-groups \(G\) \cite{Ne2}
with representatives \(G_\rho^{(m,n)}(\alpha,\beta,\gamma,\delta)\)
of order \(3^n\) and class \(m-1\),
which satisfy the relations (\ref{eqn:LowRel})
with systems of exponents \((\alpha,\beta,\gamma,\delta,\rho)\),
we now apply theorem \ref{t:LowKerTrf}
to the groups of lower than second maximal class
with transfer types of section \(\mathrm{d}\),
thereby distinguishing terminal and internal nodes
on the tree of all isomorphism classes.

\begin{theorem}
\label{t:TrmInt}

Let \(G\) be a metabelian \(3\)-group
with abelianisation of type \((3,3)\)
of lower than second maximal class, \(e\ge 4\),
and satisfying the assumptions of theorem \ref{t:LowTrmTrf}
concerning the selection of the generators of \(G=\langle x,y\rangle\)
and the order of the maximal normal subgroups \(M_1,\ldots,M_4\).\\
Suppose that the transfer type of \(G\) is one of section \(\mathrm{d}\),
which implies that \(G\) has the invariant \(k=0\).

Then the position of \(G\) as a node on the directed rooted tree 
of all metabelian \(3\)-groups with abelianisation of type \((3,3)\)
determines
the exponents \(\alpha,\beta,\gamma,\delta\) in the relations (\ref{eqn:LowRel})
and the canonical multiplet \(\varkappa\) of transfer types of \(G\),
which are given by table \ref{tab:Terminal} for terminal nodes
and by table \ref{tab:Internal} for internal nodes.

\end{theorem}

\renewcommand{\arraystretch}{1.2}
\begin{table}[ht]
\caption{Transfer types of section \(\mathrm{d}\) for terminal groups}
\label{tab:Terminal}
\begin{center}
\begin{tabular}{|c|c|c|}
\hline
 Type              & \((\alpha,\beta,\gamma,\delta)\) & \(\varkappa\) \\
\hline
 \(\mathrm{d}.19\) & \((1,0,1,0)\)                    & \((4043)\)    \\
 \(\mathrm{d}.19\) & \((1,0,-1,0)\)                   & \((3043)\)    \\
 \(\mathrm{d}.23\) & \((1,0,0,0)\)                    & \((1043)\)    \\
 \(\mathrm{d}.25\) & \((0,0,1,0)\)                    & \((2043)\)    \\
 \(\mathrm{d}.25\) & \((0,0,-1,0)\)                   & \((2043)\)    \\
\hline
\end{tabular}
\end{center}
\end{table}

\renewcommand{\arraystretch}{1.2}
\begin{table}[ht]
\caption{Transfer types of section \(\mathrm{d}^\ast\) for internal groups}
\label{tab:Internal}
\begin{center}
\begin{tabular}{|c|c|c|}
\hline
 Type                   & \((\alpha,\beta,\gamma,\delta)\) & \(\varkappa\) \\
\hline
 \(\mathrm{d}^\ast.19\) & \((0,1,0,1)\)                    & \((0443)\)    \\
 \(\mathrm{d}^\ast.19\) & \((0,-1,0,1)\)                   & \((0343)\)    \\
 \(\mathrm{d}^\ast.23\) & \((0,0,0,1)\)                    & \((0243)\)    \\
 \(\mathrm{d}^\ast.25\) & \((0,1,0,0)\)                    & \((0143)\)    \\
 \(\mathrm{d}^\ast.25\) & \((0,-1,0,0)\)                   & \((0143)\)    \\
\hline
\end{tabular}
\end{center}
\end{table}

\begin{proof}

The systems \((\alpha,\beta,\gamma,\delta)\) of relational exponents
for groups of lower than second maximal class
with transfer types of section \(\mathrm{d}\)
are given in the appendix of Nebelung's thesis
\cite[p.34, p.36, pp.60--68]{Ne2}.

If we observe that \(k=0\), and thus \(\rho=0\),
for the transfer types of section \(\mathrm{d}\),
then the multiplets \(\varkappa\) of transfer types
can be obtained immediately with the aid of theorem \ref{t:LowKerTrf}.
\end{proof}

In the case of a quadratic base field,
terminal and internal groups \(G\)
with transfer types of section \(\mathrm{d}\)
can be distinguished by the parity of the index \(m\) of nilpotency of \(G\).

\begin{theorem}
\label{t:QdrTrmInt}

Let \(K=\mathbb{Q}(\sqrt{D})\) be a quadratic base field
with elementary abelian bicyclic \(3\)-class group of type \((3,3)\)
and second \(3\)-class group \(G=\mathrm{Gal}(\mathrm{F}_3^2(K)\vert K)\)
having one of the transfer types of section \(\mathrm{d}\).
Then the following statements hold.

\begin{enumerate}

\item
The group \(G\) is represented
by a terminal node on the tree,
if and only if the index of nilpotency \(m\ge 6\)
and the invariant \(e\ge 4\) are even.

\item
The group \(G\) is represented
by an internal node on the tree,
if and only if the index of nilpotency \(m\ge 7\)
and the invariant \(e\ge 5\) are odd.

\item
\(K\) must be a real quadratic field
and \(G\) must be of coclass \(\mathrm{cc}(G)\ge 3\).

\end{enumerate}

\end{theorem}

\begin{proof}

In this proof, we assume that
the declarations of theorem \ref{t:LowTrmTrf} concerning
the selection of the generators of \(G=\langle x,y\rangle\)
and the order of the maximal normal subgroups \(M_1,\ldots,M_4\)
are satisfied.
However, we point out that these assumptions do not appear
in the statement of theorem \ref{t:QdrTrmInt}.

We denote
by \(N_1,\ldots,N_4\) the four unramified cyclic cubic extension fields of \(K\),
by \(L_1,\ldots,L_4\) their non-Galois absolutely cubic subfields,
and we make use of proposition 5.4 in \cite{Ma},
which concerns the parity of the \(3\)-exponents
of the \(3\)-class numbers of \(N_1,\ldots,N_4\).

\begin{enumerate}

\item
In the case of a terminal group \(G\)
with transfer type \(\mathrm{d}\), we have\\
\(\varkappa(2)=0\), by theorem \ref{t:TrmInt}, and thus a total principalisation in \(N_2\).
Therefore, \(3^e=\mathrm{h}_3(N_2)=\mathrm{h}_3(L_2)^2\)
with even exponent \(e\ge 4\).\\
Further, \(1\le\varkappa(1)\le 4\) implies a partial principalisation in \(N_1\), and thus 
\(3^{m-1}=\mathrm{h}_3(N_1)=3\cdot\mathrm{h}_3(L_1)^2\), since \(k=0\),
with odd exponent \(m-1\) resp. even index of nilpotency \(m\ge 6\),
because \(m=4\) implies \(e\le m-1=3\).

\item
In the case of an internal group \(G\)
with transfer type \(\mathrm{d}^\ast\), we have\\
\(1\le\varkappa(2)\le 4\), by theorem \ref{t:TrmInt}, and thus a partial principalisation in \(N_2\).
Consequently, \(3^e=\mathrm{h}_3(N_2)=3\cdot\mathrm{h}_3(L_2)^2\)
with odd exponent \(e\ge 5\),
since for \(e=3\) only terminal groups with transfer type \(\mathrm{d}\) are possible, by \cite{Ne2}.\\
Further, \(\varkappa(1)=0\) implies a total principalisation in \(N_1\), and thus 
\(3^{m-1}=\mathrm{h}_3(N_1)=\mathrm{h}_3(L_1)^2\), since \(k=0\),
with even exponent \(m-1\) resp. odd index of nilpotency \(m\ge 7\),
because \(m=5\) implies \(e\le m-1=4\).

\item
Since a total principalisation occurs in \(N_2\) resp. \(N_1\)
for transfer types in both sections \(\mathrm{d}\) and \(\mathrm{d}^\ast\),
the base field \(K\) must be real quadratic, by proposition 5.3 in \cite{Ma}.
And since the invariant \(e\) has turned out to be at least equal to four,
the group \(G\) must be of coclass \(\mathrm{cc}(G)\ge 3\).
\end{enumerate}
\end{proof}

\begin{example}
\label{ex:TrfTypD}

Occurrences of groups with transfer types in section \(\mathrm{d}\)
as second \(3\)-class group \(G=\mathrm{Gal}(\mathrm{F}_3^2(K)\vert K)\)
of real quadratic number fields \(K=\mathbb{Q}(\sqrt{D})\)
are extremely rare.
Among the \(2\,576\) quadratic fields
with \(3\)-class group of type \((3,3)\)
in the range \(0<D<10^7\) of discriminants,
there are only three cases with such groups \cite[Ex.6.3, Tab.4]{Ma}.
We have

\begin{itemize}

\item
a terminal group of transfer type \(\mathrm{d}.23\) with \(e=4\), \(m=6\), \(n=8\) for \(D=1\,535\,117\),

\item
a terminal group of transfer type \(\mathrm{d}.19\) with \(e=4\), \(m=6\), \(n=8\) for \(D=2\,328\,721\),

\item
an internal group of transfer type \(\mathrm{d}^\ast.25\) with \(e=5\), \(m=7\), \(n=10\) for \(D=8\,491\,713\).

\end{itemize}
\end{example}


\subsection{Combinatorially possible transfer types of \(3\)-groups}
\label{ss:TriadicTrf}

In this subsection, we arrange all combinatorially possible
\(S_4\)-orbits of the \(5^4\) quadruplets \(\varkappa\in\lbrack 0,4\rbrack^4\)
by increasing cardinality of the image and decreasing number of fixed points.
Table \ref{tab:TriPrtTrf} shows the partial quadruplets
and table \ref{tab:TriTotTrf} the total quadruplets as possible transfer types
of metabelian \(3\)-groups \(G\) with \(G/\gamma_2(G)\) of type \((3,3)\),
resp. principalisation types
of base fields \(K\) with \(3\)-class group \(\mathrm{Cl}_3(K)\) of type \((3,3)\).
The orbits are divided into sections, denoted by letters,
and identified by ordinal numbers.

We denote by \(o(\varkappa)=(\vert\varkappa^{-1}\lbrace i\rbrace\vert)_{0\le i\le 4}\)
the family of occupation numbers of the selected orbit representative \(\varkappa\) and by
\(F=\lbrace 1\le i\le 4\mid\varkappa(i)=i\rbrace\) the set of fixed points of \(\varkappa\).
In the characterising property, needed for equal numbers of fixed points, let
\(I=\lbrace\varkappa(i)\mid 1\le i\le 4\rbrace\) be the image of \(\varkappa\),
\(D=\varkappa^{-1}(o(\varkappa)^{-1}\lbrace 2\rbrace)=\lbrace i,j\rbrace\)
the preimage of a value occupied twice by \(\varkappa\),
if \(\varkappa(i)=\varkappa(j)\) for \(1\le i<j\le 4\), and
\(Z=\varkappa^{-1}\lbrace 0\rbrace\) the preimage of zero under \(\varkappa\).

If an orbit can be realised as a transfer type,
then a suitable \(3\)-group \(G\) is given,
according to theorem \ref{t:MaxKerTrf} and \cite[p.208, Satz 6.14]{Ne}.
In \cite[p.80]{Ma1}, the symbolic order has been given instead,
that is the ideal of bivariate polynomials which annihilate the main commutator of \(G\).

In table \ref{tab:TriPrtTrf},
the coarse classification into sections \(\mathrm{A}\) to \(\mathrm{H}\)
is due to Scholz and Taussky \cite{SoTa}.
The identification by ordinal numbers \(1\) to \(19\)
and the set theoretical characterisation
has been added in \cite[p.80]{Ma1}.

\renewcommand{\arraystretch}{1.2}
\begin{table}[ht]
\caption{The \(19\) \(S_4\)-orbits of quadruplets \(\varkappa\in\lbrack 1,4\rbrack^4\) with \(\nu=0\)}
\label{tab:TriPrtTrf}
\begin{center}
\begin{tabular}{|rr|c|cccc|c|}
\hline
      &     & repres.       & occupation       & fixed   & charact.                       & cardinality           & realising                      \\
 Sec. & Nr. & of orbit      & numbers          & points  & property                       & of orbit              & \(3\)-group                    \\
      &     & \(\varkappa\) & \(o(\varkappa)\) & \(|F|\) &                                & \(|\varkappa^{S_4}|\) & \(G\)                          \\
\hline
    A &   1 & \((1111)\)    & \((04000)\)      & \(1\)   & constant                       &   \(4\)               & \(G_0^{(3)}(0,1)\) \\
\hline
    B &   2 & \((1211)\)    & \((03100)\)      & \(2\)   & almost                         &  \(12\)               & impossible                     \\
    B &   3 & \((1112)\)    & \((03100)\)      & \(1\)   & con-                           &  \(24\)               & impossible                     \\
    H &   4 & \((2111)\)    & \((03100)\)      & \(0\)   & stant                          &  \(12\)               & \(G_1^{(5,6)}(1,1,1,1)\)                     \\
\hline
    D &   5 & \((1212)\)    & \((02200)\)      & \(2\)   &                                &  \(12\)               & \(G_0^{(4,5)}(1,1,-1,1)\)        \\
    E &   6 & \((1122)\)    & \((02200)\)      & \(1\)   &                                &  \(12\)               & \(G_0^{(m,m+1)}(1,-1,1,1)\) \\
    F &   7 & \((2112)\)    & \((02200)\)      & \(0\)   &                                &  \(12\)               & \(G_0^{(m,m+e-2)}(1,1,-1,1)\) \\
\hline
    E &   8 & \((1231)\)    & \((02110)\)      & \(3\)   &                                &  \(12\)               & \(G_0^{(m,m+1)}(1,0,-1,1)\) \\
    E &   9 & \((1213)\)    & \((02110)\)      & \(2\)   &                                &  \(24\)               & \(G_0^{(m,m+1)}(0,0,1,1)\) \\
    D &  10 & \((1123)\)    & \((02110)\)      & \(1\)   & \(D\setminus F\subset I\)      &  \(24\)               & \(G_0^{(4,5)}(0,0,-1,1)\)        \\
    F &  11 & \((1321)\)    & \((02110)\)      & \(1\)   & \(D\setminus F\not\subset I\)  &  \(12\)               & \(G_0^{(m,m+e-2)}(1,1,0,0)\) \\
    F &  12 & \((3211)\)    & \((02110)\)      & \(1\)   & \(D\cap F=\emptyset\)          &  \(24\)               & \(G_0^{(m,m+e-2)}(1,1,0,-1)\) \\
    F &  13 & \((2113)\)    & \((02110)\)      & \(0\)   & \(D\subset I\)                 &  \(24\)               & \(G_0^{(m,m+e-2)}(1,1,-1,0)\) \\
    E &  14 & \((2311)\)    & \((02110)\)      & \(0\)   & \(D\not\subset I\)             &  \(24\)               & \(G_0^{(m,m+1)}(0,-1,1,1)\) \\
\hline
    C &  15 & \((1234)\)    & \((01111)\)      & \(4\)   & identity                       &   \(1\)               & impossible                     \\
    G &  16 & \((2134)\)    & \((01111)\)      & \(2\)   & transposition                  &   \(6\)               & \(G_1^{(7,8)}(1,0,0,1)\)                     \\
    C &  17 & \((1342)\)    & \((01111)\)      & \(1\)   & \(3\)-cycle                    &   \(8\)               & impossible                     \\
    C &  18 & \((2341)\)    & \((01111)\)      & \(0\)   & \(4\)-cycle                    &   \(6\)               & impossible                     \\
    G &  19 & \((2143)\)    & \((01111)\)      & \(0\)   & \(2\) disj. transp.            &   \(3\)               & \(G_1^{(5,6)}(0,-1,-1,0)\)                     \\
\hline
      &     &               &                  &         &                  Total number: & \(256\)               &                                \\
\hline
\end{tabular}
\end{center}
\end{table}

\noindent
The \(3\)-groups of sections \(\mathrm{A}\) and \(\mathrm{D}\) are determined uniquely.
They are
sporadic groups outside the periodic parts on the coclass graphs \(\mathcal{G}(3,1)\) resp. \(\mathcal{G}(3,2)\) \cite{As,LgMk,EiLg,DEF},
stem groups of isoclinism families \(\Phi_s\) in the sense of P. Hall \cite{Hl},
and coincide with the following groups in the listing of R. James \cite[p.618 ff]{Ja}:
\(G_0^{(3)}(0,1)=\Phi_2(21)\) (extra special),
\(G_0^{(4,5)}(1,1,-1,1)=\Phi_6(221)_a\),
\(G_0^{(4,5)}(0,0,-1,1)=\Phi_6(221)_{c_2}\).
The index of nilpotency for groups of section \(\mathrm{E}\) is \(m\ge 5\).
For the groups of section \(\mathrm{F}\) we have \(m\ge 5\) and \(e\ge 4\).
The groups given for sections \(\mathrm{G}\) and \(\mathrm{H}\)
occur as second \(3\)-class groups of quadratic fields,
but they are only special cases
of families with two infinitely increasing parameters \(m\ge 4\), \(e\ge 3\),
and \(0\le k\le 1\).
The smallest members of these families are the stem groups
\(G_0^{(4,5)}(0,-1,-1,0)=\Phi_6(221)_{d_0}\) and
\(G_0^{(4,5)}(1,1,1,1)=\Phi_6(221)_{b_1}\),
but they do not occur for quadratic base fields.

In table \ref{tab:TriTotTrf},
the coarse classification into sections \(\mathrm{a}\) to \(\mathrm{e}\)
is due to \cite{Ne}.
Additionally, we give
an identification by ordinal numbers \(1\) to \(26\)
and a set theoretical characterisation.

\newpage

\renewcommand{\arraystretch}{1.2}
\begin{table}[ht]
\caption{The \(26\) \(S_4\)-orbits of quadruplets \(\varkappa\in\lbrack 0,4\rbrack^4\setminus\lbrack 1,4\rbrack^4\) with \(1\le\nu\le 4\)}
\label{tab:TriTotTrf}
\begin{center}
\begin{tabular}{|rr|c|cccc|c|}
\hline
      &     & repres.       & occupation       & fixed   & charact.                       & cardinality           & realising                      \\
 Sec. & Nr. & of orbit      & numbers          & points  & property                       & of orbit              & \(3\)-group                    \\
      &     & \(\varkappa\) & \(o(\varkappa)\) & \(|F|\) &                                & \(|\varkappa^{S_4}|\) & \(G\)                          \\
\hline
    a &   1 & \((0000)\)    & \((40000)\)      & \(0\)   & constant                       &                 \(1\) & \(G_1^{(m)}(0,\pm 1)\), \(m\ge 5\) \\
\hline
    a &   2 & \((1000)\)    & \((31000)\)      & \(1\)   &                                &                 \(4\) & \(G_0^{(m)}(0,1)\), \(m\ge 4\) \\
    a &   3 & \((0100)\)    & \((31000)\)      & \(0\)   &                                &                \(12\) & \(G_0^{(m)}(\pm 1,0)\), \(m\ge 4\) \\
\hline
    e &   4 & \((1100)\)    & \((22000)\)      & \(1\)   &                                &                \(12\) & impossible                    \\
    e &   5 & \((0110)\)    & \((22000)\)      & \(0\)   &                                &                \(12\) & impossible                    \\
\hline
    e &   6 & \((1200)\)    & \((21100)\)      & \(2\)   &                                &                 \(6\) & impossible                    \\
    e &   7 & \((1020)\)    & \((21100)\)      & \(1\)   &                                &                \(24\) & impossible                    \\
    e &   8 & \((0012)\)    & \((21100)\)      & \(0\)   & \(D\subset I\)                 &                \(12\) & impossible                    \\
    e &   9 & \((0120)\)    & \((21100)\)      & \(0\)   & \(|D\cap I|=1\)                &                \(24\) & impossible                    \\
    b &  10 & \((2100)\)    & \((21100)\)      & \(0\)   & \(D\cap I=\emptyset\)          &                 \(6\) & \(G_1^{(6,8)}(0,0,0,0)\)                     \\
\hline
    e &  11 & \((1110)\)    & \((13000)\)      & \(1\)   &                                &                \(12\) & impossible                    \\
    e &  12 & \((0111)\)    & \((13000)\)      & \(0\)   &                                &                 \(4\) & impossible                    \\
\hline
    e &  13 & \((1210)\)    & \((12100)\)      & \(2\)   &                                &                \(24\) & impossible                    \\
    e &  14 & \((1120)\)    & \((12100)\)      & \(1\)   & \(D\setminus F\subset I\)      &                \(24\) & impossible                    \\
    e &  15 & \((1012)\)    & \((12100)\)      & \(1\)   & \(D\setminus F\not\subset I\)  &                \(24\) & impossible                    \\
    e &  16 & \((0211)\)    & \((12100)\)      & \(1\)   & \(D\cap F=\emptyset\)          &                \(12\) & impossible                    \\
    e &  17 & \((0112)\)    & \((12100)\)      & \(0\)   & \(|D\cap I|=1,\ Z\subset I\)   &                \(24\) & impossible                    \\
    c &  18 & \((2011)\)    & \((12100)\)      & \(0\)   &\(D\cap I=\emptyset,\ Z\subset I\)&              \(12\) & \(G_0^{(m,m+1)}(0,-1,0,1)\) \\
    d &  19 & \((2110)\)    & \((12100)\)      & \(0\)   & \(Z\not\subset I\)             &                \(24\) & \(G_0^{(m,m+e-2)}(1,0,1,0)\) \\
\hline
    e &  20 & \((1230)\)    & \((11110)\)      & \(3\)   & identity with \(0\)            &                 \(4\) & impossible                    \\
    c &  21 & \((1203)\)    & \((11110)\)      & \(2\)   &                                &                \(12\) & \(G_0^{(m,m+1)}(0,0,0,1)\) \\
    e &  22 & \((1023)\)    & \((11110)\)      & \(1\)   & \(Z\subset I\)                 &                \(24\) & impossible                    \\
    d &  23 & \((1320)\)    & \((11110)\)      & \(1\)   & \(Z\not\subset I\)             &                \(12\) & \(G_0^{(m,m+e-2)}(1,0,0,0)\) \\
    e &  24 & \((0123)\)    & \((11110)\)      & \(0\)   & \(4\)-cycle with \(0\)         &                \(24\) & impossible                    \\
    d &  25 & \((0321)\)    & \((11110)\)      & \(0\)   & \(2\) disj. transp. with \(0\) &                \(12\) & \(G_0^{(m,m+e-2)}(0,1,0,0)\) \\
    e &  26 & \((2310)\)    & \((11110)\)      & \(0\)   & \(Z\not\subset I\)             &                 \(8\) & impossible                    \\
\hline
      &     &               &                  &         &                  Total number: & \(369=625-256\)       &                                \\
\hline
\end{tabular}
\end{center}
\end{table}

\noindent
The groups with transfer type \(\mathrm{a}.2\) (resp. \(\mathrm{a}.3\))
form families,
starting with the sporadic stem groups
\(G_0^{(4)}(0,1)=\Phi_3(211)_a\)
(resp. \(G_0^{(4)}(1,0)=\Phi_3(211)_{b_1}=\mathrm{Syl}_3\mathrm{A}_9\)
and \(G_0^{(4)}(-1,0)=\Phi_3(211)_{b_2}\)),
and continuing with the stem groups
\(G_0^{(5)}(0,1)=\Phi_9(2111)_a\)
(resp. \(G_0^{(5)}(1,0)=\Phi_9(2111)_{b_0}\))
and
\(G_0^{(6)}(0,1)=\Phi_{35}(21111)_a\)
(resp. \(G_0^{(6)}(1,0)=\Phi_{35}(21111)_{b_0}\)
and \(G_0^{(6)}(-1,0)=\Phi_{35}(21111)_{b_1}\))
on the periodic part of the coclass graph \(\mathcal{G}(3,1)\).
Transfer type \(\mathrm{a}.1\) is additionally realised by the groups
\(C(3)\times C(3)\),
\(G_0^{(m)}(0,0)\) with \(m\ge 3\),
and \(G_1^{(m)}(0,0)\) with \(m\ge 5\).
The smallest groups with \(k=1\) are
\(G_1^{(5)}(0,0)=\Phi_{10}(1^5)\),
\(G_1^{(5)}(0,1)=\Phi_{10}(2111)_{a_0}\), and
\(G_1^{(5)}(0,-1)=\Phi_{10}(2111)_{a_1}\).
The index of nilpotency for groups of section \(\mathrm{c}\) is \(m\ge 4\).
The smallest members of these two families are the stem groups
\(G_0^{(4,5)}(0,-1,0,1)=\Phi_6(221)_{d_1}\) and
\(G_0^{(4,5)}(0,0,0,1)=\Phi_6(221)_{c_1}\).
For groups of section \(\mathrm{d}\) we have \(m\ge 5\) and \(e\ge 3\).
The group given in section \(\mathrm{b}\)
occurs as second \(3\)-class group of real quadratic fields,
but it is only a special case
of a family with two infinitely increasing parameters \(m\ge 4\), \(e\ge 3\),
and \(0\le k\le 1\).
The smallest member is \(G_0^{(4,5)}(0,0,0,0)=\Phi_6(1^5)\).



\newpage

\begin{thebibliography}{XX}
%
\bibitem{Ar1}
E. Artin,
Beweis des allgemeinen Reziprozit\"atsgesetzes,
\textit{Abh. Math. Sem. Univ. Hamburg}
\textbf{5}
(1927),
353--363.
%
\bibitem{Ar2}
E. Artin,
Idealklassen in Oberk\"orpern
und allgemeines Reziprozit\"atsgesetz,
\textit{Abh. Math. Sem. Univ. Hamburg}
\textbf{7}
(1929),
46--51.
%
\bibitem{As}
J. Ascione,
\textit{On \(3\)-groups of second maximal class}
(Ph.D. Thesis,
Australian National University,
Canberra,
1979).
%
\bibitem{BeSn}
E. Benjamin and C. Snyder,
Real quadratic number fields with \(2\)-class group of type \((2,2)\),
\textit{Math. Scand.}
\textbf{76}
(1995),
161--178.
%
\bibitem{Bv}
Y. Berkovich,
\textit{Groups of prime power order, Volume 1}
(Expositions in Mathematics
\textbf{46},
de Gruyter,
Berlin,
2008).
%
\bibitem{Bl}
N. Blackburn,
On a special class of \(p\)-groups,
\textit{Acta Math.}
\textbf{100}
(1958),
45--92.
%
\bibitem{ChFt}
S. M. Chang and R. Foote,
Capitulation in class field extensions of type \((p,p)\),
\textit{Can. J. Math.}
\textbf{32}
(1980),
no.5,
1229--1243.
%
\bibitem{DEF}
H. Dietrich, B. Eick, and D. Feichtenschlager,
Investigating \(p\)-groups by coclass with GAP,
\textit{Computational group theory and the theory of groups},
45--61
(Contemp. Math.
\textbf{470},
AMS, Providence, RI,
2008).
%
\bibitem{EiLg}
B. Eick and C. Leedham-Green,
On the classification of prime-power groups by coclass,
\textit{Bull. London Math. Soc.}
\textbf{40}
(2008),
274--288.
%
\bibitem{Go}
D. Gorenstein,
\textit{Finite groups}
(Harper and Row,
New York,
1968).
%
\bibitem{Hl}
P. Hall,
The classification of prime-power groups,
\textit{J. Reine Angew. Math.}
\textbf{182}
(1940),
130--141.
%
\bibitem{Ha2}
H. Hasse,
Bericht \"uber neuere Untersuchungen und Probleme aus der
Theorie der algebraischen Zahlk\"orper. Teil II: Reziprozit\"atsgesetz,
\textit{Jber. der DMV}
\textbf{6}
(1930),
1--204.
%
\bibitem{Hi}
D. Hilbert,
Die Theorie der algebraischen Zahlk\"orper,
\textit{Jber. der DMV}
\textbf{4}
(1897),
175--546.
%
\bibitem{Ja}
R. James,
The groups of order \(p^6\) (\(p\) an odd prime),
\textit{Math. Comp.}
\textbf{34}
(1980),
nr. 150,
613--637.
%
\bibitem{Ki1}
H. Kisilevsky,
Some results related to Hilbert's theorem \(94\),
\textit{J. Number Theory}
\textbf{2}
(1970),
199--206.
%
\bibitem{Ki2}
H. Kisilevsky,
Number fields with class number congruent to \(4\) mod \(8\)
and Hilbert's theorem \(94\),
\textit{J. Number Theory}
\textbf{8}
(1976),
271--279.
%
\bibitem{LgMk}
C. R. Leedham-Green and S. McKay,
\textit{The structure of groups of prime power order},
London Math. Soc. Monographs, New Series,
\textbf{27},
Oxford Univ. Press,
2002.
%
\bibitem{Ma1}
D. C. Mayer,
Principalization in complex \(S_3\)-fields,
\textit{Congressus Numerantium}
\textbf{80}
(1991),
73--87
(Proceedings of the Twentieth Manitoba Conference
on Numerical Mathematics and Computing,
Winnipeg, Manitoba, Canada, 1990).
%
\bibitem{Ma}
D. C. Mayer,
The second \(p\)-class group of a number field,
\textit{Int. J. Number Theory}
(2011).
%
\bibitem{Ma2}
D. C. Mayer,
\textit{Principalisation algorithm via class group structure}
(Preprint 2011).
%
\bibitem{Mi}
R. J. Miech,
Metabelian \(p\)-groups of maximal class,
\textit{Trans. Amer. Math. Soc.}
\textbf{152}
(1970),
331--373.
%
\bibitem{My}
K. Miyake,
Algebraic investigations of Hilbert's Theorem \(94\),
the principal ideal theorem and the capitulation problem,
\textit{Expo. Math.}
\textbf{7}
(1989),
289--346.
%
\bibitem{Ne}
B. Nebelung,
\textit{Klassifikation metabelscher \(3\)-Gruppen
mit Faktorkommutatorgruppe vom Typ \((3,3)\)
und Anwendung auf das Kapitulationsproblem}
(Inauguraldissertation, Band 1,
Universit\"at zu K\"oln,
1989).
%
\bibitem{Ne2}
B. Nebelung,
\textit{Anhang zu
Klassifikation metabelscher \(3\)-Gruppen
mit Faktorkommutatorgruppe vom Typ \((3,3)\)
und Anwendung auf das Kapitulationsproblem}
(Inauguraldissertation, Band 2,
Universit\"at zu K\"oln,
1989).
%
\bibitem{SoTa}
A. Scholz und O. Taussky,
Die Hauptideale der kubischen Klassenk\"orper
imagin\"ar quadratischer Zahlk\"orper:
ihre rechnerische Bestimmung
und ihr Einflu\ss\ auf den Klassenk\"orperturm,
\textit{J. Reine Angew. Math.}
\textbf{171}
(1934),
19--41.
%
\bibitem{Ta}
O. Taussky,
A remark concerning Hilbert's Theorem \(94\),
\textit{J. Reine Angew. Math.}
\textbf{239/240}
(1970),
435--438.
%
\end{thebibliography}
\end{document}